\documentclass[10pt]{amsart}
\allowdisplaybreaks

\usepackage[a4paper]{geometry}
\usepackage{amsfonts}
\usepackage{amsmath}
\usepackage{amssymb}
\usepackage{amsthm}
\usepackage{amsbsy}
\usepackage{xcolor}
\usepackage{nicematrix}

\usepackage{graphicx}
\usepackage{caption}
\usepackage{subcaption}
\usepackage{float}
\usepackage{enumerate}
\usepackage{multicol}
\usepackage{mathrsfs}
\usepackage{array}
\usepackage[all,cmtip]{xy}
\usepackage[color=blue!20,textsize=footnotesize]{todonotes}
\usepackage{tikz}
\usetikzlibrary{calc,fit,patterns,decorations.markings,matrix,3d,arrows.meta}

\newcounter{dummy} \numberwithin{dummy}{section}
\newtheorem{theorem}[dummy]{Theorem}
\newtheorem{corollary}[dummy]{Corollary}
\newtheorem{lemma}[dummy]{Lemma}
\newtheorem{definition}[dummy]{Definition}
\newtheorem{proposition}[dummy]{Proposition}
\theoremstyle{remark}
\newtheorem{remark}[dummy]{Remark}
\newtheorem{example}[dummy]{Example}

\newcommand{\calH}{\mathcal{H}}
\newcommand{\calT}{\mathcal{T}}

\newcommand{\frakR}{\mathfrak{R}}

\newcommand{\R}{\mathbb{R}}

\DeclareMathOperator{\Sym}{Sym}

\DeclareMathOperator{\id}{id}

\DeclareMathOperator{\spn}{span}
\DeclareMathOperator{\tr}{tr}

\DeclareMathOperator{\E}{E}
\DeclareMathOperator{\SO}{SO}
\DeclareMathOperator{\Ort}{O}

\DeclareMathOperator{\so}{\mathfrak{so}}

\newcommand{\frakg}{\mathfrak{g}}

\newcommand{\MG}{\mathrm{MG}}

\newcommand{\DV}{\overrightarrow{\rm Deg}}
\newcommand{\hDV}{\widehat{\rm Deg}}
\newcommand{\NPDO}{\mathrm{NPDO}}

\numberwithin{equation}{section}

\title{Equivariant nonlinear partial differential operators on constant curvature spaces}
\author[F. Ballerin and E. Grong]{Francesco Ballerin and Erlend Grong}
\date{}

\address{University of Bergen, Department of Mathematics, P.O.~Box 7803, 5020 Bergen, Norway}
\email{francesco.ballerin@uib.no}
\email{erlend.grong@uib.no}

\thanks{The first author is supported by the grant GeoProCo from the Trond Mohn Foundation - Grant TMS2021STG02 (GeoProCo).}



\begin{document}

\begin{abstract}
Motivated by PDE-learning, we give a classifying space for nonlinear partial differential operators equivariant under the action of isometries. This classifying space is for operators defined on Riemannian model spaces. The nonlinear operators we are considering are those that can be written as a polynomial in linear operators. We show that the classifying space for such operators can be realized as the vector space spanned by equivalence-classes of multigraphs. We also illustrate how this realization can help us discover non-trivial linear dependence relations between nonlinear differential operators relative to the dimension of the manifold. Finally, we give some comments on operators equivariant under the identity component of the isometry group and under isometry groups of sub-Riemannian model spaces.
\end{abstract}

\maketitle


\section{Introduction}

Characterizing the class of differential operators that are equivariant for the action induced by the isometry group of a Riemannian manifold is a fundamental problem in differential geometry, with direct implications wherever symmetry is a modeling constraint. In this paper we construct a universal object for nonlinear differential operators on the Riemannian model spaces that are equivariant with respect to all isometries. By a result of Helgason~\cite{helgason1959differential} from 1959, all isometry-equivariant linear differential operators are polynomials in the Laplace-Beltrami operator $\Delta$, giving a one-to-one correspondence between such operators and polynomials in one variable. Helgason proved this statement for the class of two-point homogeneous spaces, which includes the constant curvature, simply connected spaces. For the latter, we will show that the classifying space for nonlinear polynomial  operators is the real vector space spanned over equivalence classes of multigraphs. 

Let $M$ be a given manifold of dimension $d \geq 2$. Let $L_1, \dots, L_n$ be linear partial differential operators of respective orders $r_1, \dots, r_n$ on $M$, and define an operator
\begin{equation} \label{Pdiff} P: f \mapsto (L_1 f) \cdot (L_2f) \cdots (L_n f), \qquad f \in C^\infty(M). \end{equation}
We then say that $P$ is a \emph{nonlinear polynomial operator of order $k = \max \{r_1, \dots, r_n\}$, total order $p = r_1 + \cdots +r_n$ and polynomial degree $n$.}
We include the possibility where $P(f) =\psi \in C^\infty(M)$ is equal to a function independent of $f$ as an operator with order, total order and polynomial degree all zero. We extend this definition by linearity, stating that $P = \sum_{l=1}^{l_{\max}} P_l$ is a nonlinear polynomial differential operator if it is a sum of differential operator $P_1, \dots, P_{l_{\max}}$ on the form as in \eqref{Pdiff}. We say that $P$ has respectively order $k$, total order $p$ and polynomial degree~$n$ if these are respectively the maximal order, total order and polynomial degree over all the operators $P_1,\dots,P_{l_{\max}}$. We write $\NPDO(M)$ for the vector space of such nonlinear polynomial differential operators, and write $\NPDO^p(M)$ for the subspace of operators of total order at most $p$.

Suppose that the manifold $M$ under consideration is endowed with a Riemannian metric $g$ with Levi-Civita connection $\nabla$.
%
%
Let $G$ be the isometry group of $(M,g)$. We say that $P \in \NPDO(M)$ is \emph{$G$-equivariant} if
\[P(f \circ \varphi) = (Pf)\circ \varphi, \qquad \text{ for all $\varphi \in G, f\in C^\infty(M)$}.\]
We write $\NPDO(M)^G$ for the subspace of such equivariant operators. The larger the isometry group $G$ of $(M,g)$ is, the more restrictions we have for an operator to be $G$-equivariant. To introduce a universal object for manifolds with maximal isometry groups, let $\MG$ denote the collection of all isomorphism classes of multigraphs with a finite collection of vertices and edges, see Section~\ref{sec:MG} for more details. For each isomorphism class $[\gamma] \in \MG$ of the multigraph $\gamma$ with vertices $V[\gamma]$ and edges $E[\gamma]$, introduce $N_\gamma \in \NPDO(M)$ according to the following rules: for each vertex $v \in V[\gamma]$, define $\nabla^v = \nabla^{\deg(v),\Sym}$ as the symmetrized covariant derivative of order equal to that of $v$ and let $\nabla^{V[\gamma]} f= \otimes_{v\in V[\gamma]} \nabla^v f$ for any $f \in C^\infty(M)$. We then define
$$N_\gamma f= \tr_{E[\gamma]} \nabla^{V[\gamma]} f,$$
where $\tr_{E[\gamma]}$ denotes the operation of taking one trace for every edge in $E[\gamma]$, such that if $e = \{ v_1, v_2\} \in E[\gamma]$ is an edge between $v_1$ and $v_2$, we take a corresponding trace between $\nabla^{v_1} f$ and $\nabla^{v_2}f$. The order of taking these traces does not matter, as the symmetrized covariant derivatives make the operations commute. Furthermore, $\nabla^{V[\gamma]} f$ is a tensor whose order equals to the sum of all of the degrees of the vertices (twice the number of edges), so the process of taking one trace for each edge leads to $N_\gamma f$ having values in scalars. If $\gamma = \bullet^n$ contains no edges, but with $n \geq 1$ vertices, then we use the convention that $\tr_E$ is the identity, so that $N_\gamma = \nabla^{V[\gamma]} f = f^n$. We adopt the convention that $\nabla^{0,\Sym}f =f$. Finally, if $\gamma=\bullet^0$ is the null graph, i.e., the graph with no vertices, we use the convention that $N_{\bullet^0}f = 1$. We highlight the following properties of this definition.
\begin{enumerate}[(i)]
\item If $\gamma = \gamma_1 \cup \gamma_2$ is the disjoint of union two multigraphs, then $N_{\gamma}f = (N_{\gamma_1}f) \cdot  (N_{\gamma_2} f)$.
\item If $\gamma$ has $n$ vertices, $p$ edges and $k$ is the maximal degree of all of the vertices, then $N_\gamma$ has order $k$, total order $2p$ and polynomial degree $n$. In particular, the total order is always even.
\end{enumerate}
We can verify that $N_{\gamma}$ does not depend on the representative from the isomorphism class $[\gamma]$, see Section~\ref{sec:LandtildeL} for details. For examples of these operators, see Section~\ref{sec:Examples}. With this notation introduced, we are ready to present our main result.
\begin{theorem} \label{th:main}
Let $(M,g)$ be a simply connected, Riemannian model space, i.e., a simply connected Riemannian manifold with a constant sectional curvature. Write $d \geq 2$ for its dimension and $G$ for its isometry group. Let $\spn_{\mathbb{R}} \MG= \oplus_{[\gamma] \in \MG} \mathbb{R} [\gamma]$ denote the vector space spanned by the set of isomorphism classes of multigraphs. Then the linear map
$$N: \spn_{\mathbb{R}} \MG \to \NPDO(M)^G, \qquad \text{determined by} \qquad N([\gamma]) = N_\gamma,$$
is surjective. Furthermore, if $\MG^p$ is the subset of isomorphism classes of multigraphs with at most $p$ edges and $p \leq d$ then,
$$N|_{\spn_{\mathbb{R}}\MG^p} :\spn_{\mathbb{R}}\MG^p \to \NPDO^{2p}(M)^G \qquad \text{is bijective.}$$
\end{theorem}

See Section~\ref{sec:ProofMain} for the proof, where we will use the structure of the frame bundle along with Iwahori's result on invariant tensors on Euclidean space \cite{iwahori1958some}. Interestingly, similar results have recently appeared for nonlinear operators on vector fields, using the formalism of $B$-series \cite{laurent2025universal}.
 We emphasize that the maps in the image of $N$ defined in Theorem~\ref{th:main} will be equivariant under isometries for any Riemannian manifold, but there will in general be more operators when it is not a Riemannian model space. Also observe that we are considering equivariance with respect to the full isometry group, not just its connected component, the latter of which we will give some details on in Section~\ref{sec:Lesssymmetry}.

From Theorem~\ref{th:main}, the kernel of the map $N$ always consists of linear combinations of multigraphs with strictly more than $d$ edges, meaning that the intersection of the kernel of $N$ over all possible dimensions of $M$ is equal to zero. By studying this kernel, we can discover non-trivial relations between nonlinear operators. We illustrate this application by giving the following result for differential operators of total order $6$.
\begin{theorem} \label{th:PQvanishes}
On any Riemannian manifold $(M, g)$ of dimension $d \geq 2$, and let $\Delta f = \tr \nabla^2_{\ast ,\ast } f$ denote its Laplacian. Define nonlinear differential operators $P^{\beta_0}, Q^{\beta_0}$, $\beta_0 = 0,1,2, \dots$ on $M$ by
\begin{align*}
P^{\beta_0}f & = f^{\beta_0} \cdot \left(\|\nabla f\|^2 (\Delta f)^2 - \|\nabla f\|^2 \|\nabla^2 f\|^2 + 2\|\nabla^2 f( \nabla f, \cdot )\|^2 - 2\Delta f \cdot \nabla^2f(\nabla f, \nabla f) \right), \\
Q^{\beta_0}f & = f^{\beta_0} \cdot \left((\Delta f)^3 - 3 \Delta f \| \nabla^2 f\|^2 + 2\tr \nabla^2_{\ast _1,\ast _2} f \nabla^2_{\ast _1,\ast _3} f \nabla^2_{\ast _2,\ast _3}f  \right),
\end{align*}
for any $f\in C^\infty(M)$. Then these operators are equivariant with respect to isometries. Furthermore, the operators $\{ P^{\beta_0}, Q^{\beta_0}\}$ are identically zero if and only if $d = 2$.
\end{theorem}
These operators exactly correspond to the generators of the kernel of $N|_{\spn_{\mathbb{R}} \MG^3}$ when $d=2$, see Section~\ref{sec:CounterEx} for details and proof.

Our motivation for this work comes in part from the growing field of PDE-learning, involving data-driven discovery of a partial differential operator determining the underlying dynamics. Obtaining a human-readable PDE usually involves sparse regression, where the regressors come from numerical computations of a library of candidate nonlinear partial differential operators \cite{brunton2016discovering,rudy2017data,messenger2021weak}, sometimes with an intermediate step involving neural networks \cite{both2021deepmod,stephany2022pde}. If the underlying dynamics is assumed to be preserved under isometries, then we only need isometry-equivariant candidate terms. 
Theorem~\ref{th:main} gives a complete and explicit basis for the space of nonlinear polynomial isometry-equivariant operators on model spaces with multigraphs serving as an explicit encoding. In particular, if one limits the consideration of candidate terms by polynomial degree and total order, then Theorem~\ref{th:main} shows that the number of equivariant candidate terms are asymptotically constant with respect to the dimension $d$. For example, Theorem~\ref{th:main} shows that the number of equivariant candidate terms up to total order $4$ is the same for every $d \geq 2$, while Theorem~\ref{th:PQvanishes} shows the difference between $d =2$ and higher dimensions in the case of total order 6. Including equivariance in PDE-learning is a natural continuation of its involvement in other PDE-related learning tasks, e.g., ~\cite{gao2022roteqnet, he2022neural, lagrave2022equivariant, andersdotter2024equivariant, sangalli2022differential}.

\smallskip

\paragraph{\bf Structure of the paper}
Section~\ref{sec:preliminaries} collects the necessary background on multigraphs and their encoding, the theory of invariant tensors and Iwahori's theorem for linear tensors invariant under orthogonal transformations. Section~\ref{sec:PDOs} develops the theory of nonlinear polynomial differential operators on Riemannian manifolds, introducing symmetrized covariant derivatives and the contraction machinery. Section~\ref{sec:LandtildeL} constructs the operators $N_\gamma$ associated to multigraphs in detail and proves that they are well-defined. This section also contains several examples. Section~\ref{sec:FrameBundles} introduces frame bundles, and connects the concepts of equivariant operators on a manifold and invariant maps on the frame bundle. Section~\ref{sec:EquivariantOperators} describes Riemannian model spaces and contains the proof of Theorem~\ref{th:main}. Section~\ref{sec:furtherResults} finally looks at several related problems to the main result. We look in detail at the kernel of the map $N$ when $p = 6$ and $d=2$, giving us in the end Theorem~\ref{th:PQvanishes}. This section also gives a classifying space for operators that are invariant under the connected component of the isometry group; a result which is valid for real projective spaces in addition to the Riemannian model spaces. Finally, we give some brief details on nonlinear operators invariant under isometry groups of sub-Riemannian model spaces. The appendix contains tables with more examples.

\smallskip

\paragraph{\bf Notation}
We emphasize the following points related to notation, in order to make it simpler for the reader to follow.
\begin{enumerate}[$\bullet$]
\item The dimension of the manifold $M$ will always be denoted $d$. Points are denoted as $x\in M$, while tangent vectors are denoted by $\xi \in TM$. Vector fields are denoted by $Y$, while $X$ can be a vector field or a multi-vector field.
\item $G$ will always be an isometry group of the Riemannian manifold $(M,g)$ in question and $\varphi \in G$ will be used for an individual isometry.
\item A multigraph $\gamma$ will always have $n$ vertices and $p$ edges. We will denote vertices and vertex sets as $v \in V[\gamma]$, and similarly write $e \in E[\gamma]$ for edges.
\item We denote linear differential operators by $L$, while nonlinear operators are denoted by~$P$ or~$Q$. The latter will have polynomial degree $n$ and total order $p$, or $2p$ if the order is even. This notation is chosen to reflect the correspondence with multigraphs. We use $N$, possibly with some accents, for a mapping into nonlinear operators.
\end{enumerate}
\bigskip

\paragraph{\bf Acknowledgment} We thank Gunnar Fløystad and Adrien Busnot Laurent for helpful comments.

\section{Preliminaries} \label{sec:preliminaries}

\subsection{Multigraphs} \label{sec:MG}

In this paper, multigraphs arise naturally as combinatorial objects encoding invariant contraction patterns between symmetric covariant derivatives of a function. We begin by reviewing the pertinent definitions and establishing the notation.

A \emph{multigraph} $\gamma$ consists of a finite set of vertices $V[\gamma]=\{v_1,\dots,v_n\}$ together with a finite multiset of edges $E[\gamma]=\{\{v_{i_1},v_{j_1}\},\dots,\{v_{i_p},v_{j_p}\}\}$, where each edge $e \in E[\gamma]$ is an unordered pair $\{v_1,v_2\}$ of vertices in $V[\gamma]$. In particular, multiple edges between the same pair of vertices are allowed, and we also allow \emph{loops}, i.e., edges of the form $\{v,v\}$. The reader should take note that in some of the literature multigraphs do not allow loops, and the term \textit{pseudographs} is used instead \cite{Bollobas1998, Harary1969}.
Two multigraphs are said to be \emph{isomorphic} if there exists a bijection between their vertex sets inducing a bijection between edge multisets.

The \emph{degree} $\deg(v)$ of a vertex $v \in V[\gamma]$ is defined as the number of edges incident to $v$, where each loop contributes $2$ to the degree of the vertex to which it is attached.

\begin{figure}[h]
\centering
\begin{tikzpicture}[scale=0.7]

\begin{scope}[xshift=0cm]
  \node[circle, draw, fill=black, inner sep=2pt] (A1) at (0,0) {};
  \node[circle, draw, fill=black, inner sep=2pt] (B1) at (2,0) {};
  \node[circle, draw, fill=black, inner sep=2pt] (C1) at (1,1.4) {};

  \draw (A1) -- (B1);
  \draw (C1) to[out=60,in=120,loop] ();

  \node at (1,-0.8) {(a)};
\end{scope}

\begin{scope}[xshift=5cm]
  \node[circle, draw, fill=black, inner sep=2pt] (A2) at (0,0) {};
  \node[circle, draw, fill=black, inner sep=2pt] (B2) at (1,1.4) {};
  \node[circle, draw, fill=black, inner sep=2pt] (C2) at (2,0) {};

  \draw (A2) to[out=60,in=120,loop] ();
  \draw (B2) -- (C2);

  \node at (1,-0.8) {(b)};
\end{scope}

\begin{scope}[xshift=10cm]
  \node[circle, draw, fill=black, inner sep=2pt] (A3) at (0,0) {};
  \node[circle, draw, fill=black, inner sep=2pt] (B3) at (2,0) {};
  \node[circle, draw, fill=black, inner sep=2pt] (C3) at (1,1.4) {};

  \draw (B3) to[bend left=20] (C3);
  \draw (B3) to[bend right=20] (C3);

  \node at (1,-0.8) {(c)};
\end{scope}

\end{tikzpicture}
\caption{Three multigraphs with the same number of vertices.
Graphs (a) and (b) are isomorphic, while graph (c) is not isomorphic to them.}
\end{figure}
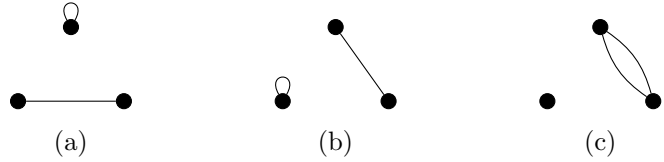

In the combinatorial study of multigraphs, one natural grading arises from the total number of edges $p=\lvert E[\gamma]\rvert$, which we will refer to as the degree of the multigraph (not to be confused with the vertex degree).
We will frequently classify multigraphs according to their \emph{degree}, counting loops with multiplicity one. We can construct increasing sets $\MG^0 \subset \MG^1 \subset \MG^2 \subset \cdots \subset \MG$ where $\MG^p$ consists of equivalence classes of multigraphs with at most $p$ edges.
While $\MG^p$ is not a finite set, the subset $\MG_\times^p$ of isomorphism classes without isolated vertices is finite, and we have a bijection $\MG^p \to \mathbb{Z}_{\geq 0} \times \MG_\times^p$. If $\bullet^{\beta_0}$ denotes the multigraph of $\beta_0 \geq 0$ isolated vertices, then this bijection from $\mathbb{Z}_{\geq 0} \times \MG_\times^p$ to $\MG^p$ is given by $(\beta_0,[\gamma]) \to [\bullet^{\beta_0} \cup \gamma]$.

\subsection{Encoding multigraphs by perfect matchings} \label{sec:UnorderedPairs}

For an integer $p \geq 1$, a \emph{perfect matching} on $\{1,\dots,2p\}$ is a partition into $p$ unordered pairs. We write $\frakR_p$ for the set of all such matchings, so each element takes the form
\[
    \rho = \bigl\{\{i_1,i_2\},\,\{i_3,i_4\},\,\dots,\,\{i_{2p-1},i_{2p}\}\bigr\}.
\]
The symmetric group $S_{2p}$ acts on $\frakR_p$ by
\[
    \sigma\cdot\rho = \bigl\{\{\sigma(i_1),\sigma(i_2)\},\,\dots,\,\{\sigma(i_{2p-1}),\sigma(i_{2p})\}\bigr\},
\]
and this action is transitive. We will adopt the convention that $\frakR_0 = \{ \emptyset\}$ contains just the empty set for $p = 0$.

The goal of this section is to show that every isomorphism class of multigraphs with $p$ edges can be represented by a pair $(\rho,\beta)$, where $\rho\in\frakR_p$ encodes an edge-pairing pattern and $\beta: \mathbb{Z}_{0\geq 0}\to\mathbb{Z}_{\geq 0}$ encodes the distribution of vertex degrees. Given a multigraph $\gamma$ with $p$ edges, an \emph{edge-end} of $\gamma$ is a formal incidence of an edge at a vertex: each non-loop edge $\{u,v\}$ contributes one edge-end at $u$ and one at $v$, while a loop at $v$ contributes two edge-ends at $v$. In either case each edge contributes exactly two edge-ends, so $\gamma$ has $2p$ edge-ends in total.

We define a \emph{parametrization} of $\gamma$ as a pair $(\rho,m)$ consisting of $\rho\in\frakR_p$ and a map $m:\{1,\dots,2p\}\to V[\gamma]$ such that
\begin{equation}\label{EdgeM}
    E[\gamma] = m(\rho) := \bigl\{\{m(i_1),m(i_2)\},\,\dots,\,\{m(i_{2p-1}),m(i_{2p})\}\bigr\}.
\end{equation}
Intuitively, $m$ assigns a vertex to each index and~$\rho$ pairs the indices into edges. Every multigraph of degree~$p$ admits a parametrization: choose any bijection from $\{1,\dots,2p\}$ to the edge-ends of~$\gamma$, let~$m(i)$ be the vertex at which edge-end $i$ lies, and let $\rho$ pair the two indices whose edge-ends belong to the same edge. Conversely, given any $\rho\in\frakR_p$ and map $m:\{1,\dots,2p\}\to V$ into any finite set $V$, equation~\eqref{EdgeM} defines a multigraph with vertex set $V$ which we will write $\tilde\Gamma(\rho,m)$.

Related to parametrizations $\gamma = \tilde \Gamma(\rho, m)$, we introduce the following concept.
\begin{definition}
A map $\beta:\mathbb{Z}_{\geq 0} \to\mathbb{Z}_{\geq 0}$ is called a \emph{degree vector} if
\[
    |\beta|_E \;:=\; \frac{1}{2}\sum_{j=1}^\infty j\cdot\beta(j)
\]
is finite and a non-negative integer. In particular, $\beta(r)$ is zero for all but finite $r \in \mathbb{Z}_{\geq 0}$. We also set $|\beta|_V := \sum_{j=1}^\infty \beta(j)$.
\end{definition}

We write $\DV$ for the set of all degree vectors. The motivation for defining degree vectors comes from multigraphs, since if we let $\beta(j)$ counts the number of vertices of degree $j$ in a multigraph, so $|\vec{\beta}|_V$ and $|\vec{\beta}|_E$ are the total number of vertices and edges, respectively. The integrality condition on $|\vec{\beta}|_E$ is automatic for multigraphs, since $\sum_j j\cdot\vec{\beta}(j)$ equals twice the number of edges by the handshaking lemma.

If $\beta$ is not identically zero, we say that it has \emph{order} $k$ if it is is the maximal number such that $\beta(k) \neq 0$. To have an efficient notation for degree vectors, we write an order $k$-degree vector as $\beta = (\beta_0, \vec{\beta})$, with $\beta_0 =\beta(0)$ and $\vec{\beta} = [\beta(1), \beta(2), \dots, \beta(k)]) =: [ \beta_1, \dots, \beta_k]$, which gives sufficient information to determine the degree vector. We also adopt the notation that $|\vec{\beta}|_E = |(0, \vec{\beta})|_E$ and $|\vec{\beta}|_E = |(0, \vec{\beta})|_E$. We emphasize that the value of $|\beta|_E$ does not depend on the value of $\beta_0$. We write $\beta =0$ when both $\beta_0$ and $\vec{\beta}$ are identically zero.

Given a degree vector $\beta$ with $|\beta|_E = p$ and $|\beta|_V = n$, we construct a canonical vertex set $V_{\beta} = \{v_1,\dots,v_n\}$ by listing vertices in non-decreasing order of degree: $v_1,\dots,v_{\beta_0}$ have degree~$0$; the next $\beta_1$ vertices have degree~$1$; and so on.
We define the map $m_{\beta}:\{1,\dots,2p\}\to V_{\beta}$ uniquely by the following conditions. Let $l_I$ be the degree of $v_I$ satisfying $\sum_{I=1}^n l_I = 2p$. Setting $s_0 = 0$ and $s_I = \sum_{J=1}^{I-1} l_J$, we require
\[
    m_{\beta}^{-1}(v_I) = \{s_I+1,\,\dots,\,s_I+l_I\} \qquad \text{for all } 1\leq I\leq n.
\]
with the set above being empty when $l_I =0$. Uniqueness follows because the fibers are prescribed by $\beta$ and are listed in the natural order $v_1 < v_2 < \cdots < v_n$. The fiber $m_{\beta}^{-1}(v_I)$ indexes the $l_I$ edge-ends at $v_I$. For any $\rho\in\frakR_p$ we define
\[
    \tilde\Gamma(\rho,\beta) = \tilde\Gamma(\rho,\beta_0,\vec{\beta}) \;:=\; \tilde\Gamma(\rho,\,m_{\beta}).
\]
It will also be convenient to have the notation $\tilde\Gamma(\rho,\vec{\beta}) := \tilde\Gamma(\rho,0,\vec{\beta})$, in particular in examples. For the special case of $p =0$, where $\rho = \emptyset$ and $\beta =0$, we define $\tilde \Gamma(\emptyset,0) = \bullet^0$, the null graph.

\begin{example}\label{ex:ParamGraph}
Let $\beta_0 =0$ and $\vec{\beta} = [2,2]$. Then $|\beta|_V = 4$ and $|\beta|_E = \tfrac{1}{2}(1\cdot 2 + 2\cdot 2) = 3$, so $p = 3$. The canonical map $m = m_{0,[2,2]}:\{1,2,3,4,5,6\}\to\{v_1,v_2,v_3,v_4\}$ is given by
\[
    m(1) = v_1,\quad m(2) = v_2,\quad m(3) = m(4) = v_3,\quad m(5) = m(6) = v_4,
\]
where $v_1,v_2$ are the two degree-$1$ vertices and $v_3,v_4$ are the two degree-$2$ vertices. For $\rho_1 = \{\{1,2\},\{3,4\},\{5,6\}\}$ and $\rho_2 = \{\{1,3\},\{2,5\},\{4,6\}\}$, the multigraphs $\tilde\Gamma(\rho_1,[2,2])$ and $\tilde\Gamma(\rho_2,[2,2])$ are shown below and are not isomorphic.

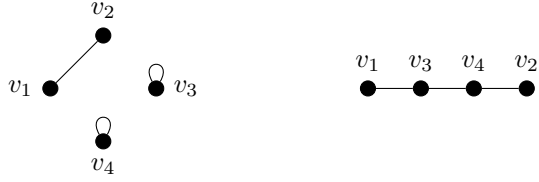
\begin{figure}[h]
\centering
\begin{tikzpicture}[scale=0.7]

\begin{scope}[xshift=0cm]
  \node[circle, draw, fill=black, inner sep=2pt, label=left:$v_1$] (V1) at (0,0) {};
  \node[circle, draw, fill=black, inner sep=2pt, label=above:$v_2$] (V2) at (1,1) {};
  \node[circle, draw, fill=black, inner sep=2pt, label=right:$v_3$] (V3) at (2,0) {};
  \node[circle, draw, fill=black, inner sep=2pt, label=below:$v_4$] (V4) at (1,-1) {};

  \draw (V1) to (V2);
  \draw (V3) to[out=60,in=120,loop] ();
  \draw (V4) to[out=60,in=120,loop] ();

\end{scope}

\begin{scope}[xshift=6cm]
  \node[circle, draw, fill=black, inner sep=2pt, label=$v_1$] (V1) at (0,0) {};
  \node[circle, draw, fill=black, inner sep=2pt, label=$v_3$] (V3) at (1,0) {};
  \node[circle, draw, fill=black, inner sep=2pt, label=$v_4$] (V4) at (2,0) {};
  \node[circle, draw, fill=black, inner sep=2pt, label=$v_2$] (V2) at (3,0) {};

  \draw (V1) to (V3);
  \draw (V3) to (V4);
  \draw (V4) to (V2);
\end{scope}

\end{tikzpicture}
\caption{Multigraphs of $\tilde \Gamma(\rho_1, [2,2])$ (left) and $\tilde \Gamma(\rho_2,[2,2])$ (right), which are not isomorphic.}
\end{figure}

\end{example}

\begin{proposition} \label{prop:ParamSurj}
Every multigraph $\gamma$ with $p$ edges is isomorphic to $\tilde\Gamma(\rho,\,\beta)$ for some degree vector $\vec{\beta} \in \DV$ with $|\beta|_E =p$.
\end{proposition}

\begin{proof}
Let $w_1, \dots, w_n$ be the vertices ordered so that $\deg(w_1) \leq \cdots \leq \deg(w_n)$. Set $l_I = \deg(w_I)$ and define $\beta(j) := \#\{I : l_I = j\}$. Then $|\beta|_V = n$ and, by the handshaking lemma, $|\beta|_E = \frac{1}{2}\sum_I l_I = p$, so $\beta$ is a degree vector with $|\beta|_E = p$.

Again set $s_0 = 0$ and $s_I = \sum_{J=1}^{I-1} l_J$ for $I =1, \dots, n$.
Choose for each $1 \leq I \leq n$ a bijection
\[
    \phi_I \colon \bigl\{\text{edge-ends of } \gamma \text{ at } w_I\bigr\}
    \;\xrightarrow{\;\sim\;}\;
    \{s_I+1, \dots, s_I+l_I\}.
\]
Since edge-ends at distinct vertices are disjoint and account for all $2p$ edge-ends of $\gamma$, the maps $\phi_I$ combine into a bijection $\phi$ from the edge-ends of $\gamma$ to $\{1,\dots,2p\}$. For each edge $e \in E[\gamma]$, let $a_e$ and $b_e$ denote the $\phi$-images of its two edge-ends, and define
\[
    \rho \;:=\; \bigl\{\{a_e,\, b_e\} : e \in E[\gamma]\bigr\}.
\]
Since each edge-end belongs to exactly one edge, every index appears in exactly one pair of $\rho$, so $\rho \in \frakR_p$.

For any edge $e = \{w_I, w_J\}$, the edge-ends of $e$ lie at $w_I$ and $w_J$, so $a_e \in \phi_I(\cdots) = m_{\beta}^{-1}(v_I)$ and $b_e \in m_{\beta}^{-1}(v_J)$. Therefore the edge multiset of $\tilde\Gamma(\rho,\beta)$ equals $\{\{v_I, v_J\} : \{w_I,w_J\} \in E[\gamma]\}$, and the map $\psi : v_I \mapsto w_I$ is an isomorphism from $\tilde\Gamma(\rho,\beta)$ to $\gamma$.
\end{proof}

In light of this proposition, we define
\[
    \Gamma(\rho,\beta) := [\tilde\Gamma(\rho,\beta)].
\]
as isomorphism classes in $\MG$. We remark that in particular $\MG^p_{\times}$ is the image of all $\rho \in \frakR_{\tilde p}$ and $(0,\vec{\beta})\in \DV$ with $|\vec{\beta}|_{E} = \tilde p$, $\tilde p \leq p$.

The parametrized multigraphs $\tilde\Gamma(\rho,\beta)$ carry natural symmetries arising from their construction, and we can identify two subgroups of $S_{2p}$ that capture them. The first, $S_{0,\beta}$, is generated by permutations within individual fibers of $m_{\beta}$: it consists of all $\sigma \in S_{2p}$ that map each fiber $m_{\beta}^{-1}(v_I)$ to itself. For such $\sigma$ one has $m_{\beta} \circ \sigma = m_{\beta}$, so every edge of $\tilde\Gamma(\sigma\cdot\rho,\beta)$ coincides with the corresponding edge of $\tilde\Gamma(\rho,\beta)$. In particular,
\[
    \tilde\Gamma(\sigma\cdot\rho,\beta) = \tilde\Gamma(\rho,\beta) \qquad \text{for all } \sigma \in S_{0,\beta}.
\]
The second, $S_{1,\beta}$, captures the freedom to swap fibers of equal size. If $v_{I_1}$ and $v_{I_2}$ have the same degree $l$, with fibers $m_{\beta}^{-1}(v_{I_s}) = \{i_s+1,\dots,i_s+l\}$ for $s=1,2$, define $\sigma_{I_1I_2} \in S_{2p}$ by
\[
    \sigma_{I_1I_2}(i_1+r) = i_2+r, \qquad
    \sigma_{I_1I_2}(i_2+r) = i_1+r \qquad (r = 1,\dots,l),
\]
and the identity on all remaining elements. Since $\sigma_{I_1I_2}$ maps $m_{\beta}^{-1}(v_{I_1})$ onto $m_{\beta}^{-1}(v_{I_2})$ and vice versa, the multigraph $\tilde\Gamma(\sigma_{I_1I_2}\cdot\rho,\beta)$ is obtained from $\tilde\Gamma(\rho,\beta)$ by relabeling $v_{I_1}$ and $v_{I_2}$. The two multigraphs need not be equal as labeled graphs, but are isomorphic via the vertex transposition of $v_{I_1}$ and $v_{I_2}$ (which is valid since both vertices have the same degree $l$). We let $S_{1,\beta}$ be the subgroup of $S_{2p}$ generated by all such transpositions $\sigma_{I_1I_2}$.

Finally, $S_{\beta}$ denotes the subgroup of $S_{2p}$ generated by both $S_{0,\beta}$ and $S_{1,\beta}$. From the above, $\Gamma(\sigma\cdot\rho,\beta) = \Gamma(\rho,\beta)$ for all $\sigma \in S_{\beta}$. Since $S_{\beta}$ with $\beta = (\beta_0, \vec{\beta})$ does not depend on $\beta_0$ by construction, we will frequently use the notation $S_{\vec{\beta}}$ Moreover, since the degree of $v_I$ in $\tilde\Gamma(\rho,\beta)$ equals $l_I$ for every $\rho$ (each fiber element appears in exactly one pair of $\rho$), any isomorphism between two parametrized multigraphs with the same $\beta$ must permute vertices within degree classes, and can therefore be realized by an element of $S_{\vec{\beta}}$. Thus the $S_{\vec{\beta}}$-orbits on $\frakR_p$ are in bijection with the isomorphism classes of multigraphs with degree vector $(0,\vec{\beta})$. 

\begin{example}
In Example~\ref{ex:ParamGraph}, the group $S_{0,[2,2]}$ is generated by transpositions $(34)$ and $(56)$, which permute within the fibers of $v_3$ and $v_4$ respectively, while $S_{1,[2,2]}$ is generated by $(12)$, swapping the fibers of $v_1$ and $v_2$, and $(35)(46)$, swapping the fibers of $v_3$ and $v_4$.
\end{example}



\subsection{Invariant tensors} \label{sec:Tensors}
Let $\calT^p_d$ be the space of covariant $p$-tensors on $\mathbb{R}^d$ that are invariant under the action of $\Ort(d)$.
In other words, we consider $p$-tensors $\tau$ such that
$$\tau(A x_1, \dots, Ax_p) = \tau(x_1, \dots, x_p), \qquad \text{for any $x_1, \dots, x_p \in \mathbb{R}^d$, $A \in \Ort(d)$}.$$
Recall the definition of $\frakR_p$ from Section~\ref{sec:UnorderedPairs}.
Observe that this set consists of $\frac{(2p)!}{p! \cdot 2^p} = (2p-1) \cdot (2p-3) \cdots 3 \cdot 1$ elements.
For every $\rho \in \mathfrak{R}_p$, define a tensor in $\calT_d^{2p}$ by
$$\tau_\rho(x_1, \dots, x_{2p}) = \prod_{\{i,j\} \in \rho} \langle x_i, x_j \rangle.$$
For the special case of $p=0$, we define $\tau_\emptyset =1$ as a constant.

For the proofs in our paper, we rely on the results below by Iwahori \cite[Theorem 1 and Remark 2]{iwahori1958some}. See also \cite[Chapter~II]{weyl1946classical}. 
\begin{lemma} \label{lemma:RhoSpans}
\begin{enumerate}[\rm (a)]
\item For $p$ odd, $\calT_d^p =0$.
\item The space $\calT_d^{2p}$ is spanned by $\tau_\rho$, $\rho \in \frakR_p$.
\item If $p\leq d$, then $\tau_\rho$, $\rho \in \frakR_p$ is a basis for $\calT_d^{2p}$.
\end{enumerate}
\end{lemma}
For the first non-trivial case, we will show explicitly in Section~\ref{sec:CounterEx} that $\tau_\rho$, $\rho \in \frakR_3$ does not form a basis for $\calT_2^6$, and show how this fact related to linear dependence relations of nonlinear differential operators.

\begin{remark} \label{re:averageVecm}
Consider the space of tensors $(\mathbb{R}^{d,\ast })^{2p}$ with $S_{2p}$ acting on elements here by $(\sigma \cdot \tau)(x_1, \dots, x_{2p}) = \tau(x_{\sigma^{-1}(1)}, \dots, x_{\sigma^{-1}(2p)})$. We then see that the symmetric product $\Sym^{2p} TM$ is the image of the averaging map $\tau \mapsto \frac{1}{(2p)!} \sum_{\sigma \in S_{2p}}(\sigma \cdot \tau)$. Similarly, for a given degree vector $\beta = (\beta_0, \vec{\beta})\in \DV$ with $|\vec{\beta}|_E = p$, define $\Sym^{\vec{\beta}} (\mathbb{R}^d)^\ast $ as the image of $\tau \mapsto \frac{1}{|S_{\vec{\beta}}|} \sum_{\sigma \in S_{\vec{\beta}}}(\sigma \cdot \tau)$. We see that $\Sym^{\vec{\beta}} (\mathbb{R}^{d})^\ast $ is spanned by $\alpha^{\odot \vec{\beta}(1)}_1 \otimes \cdots \otimes \alpha^{\odot \vec{\beta}(k)}_k$ with $\alpha_j \in \Sym^j \mathbb{R}^{d,\ast }$. An important point for later is that if $\tau \in \calT^{2p}_d$, then $\frac{1}{|S_{\vec{\beta}}|} \sum_{\sigma \in S_{\vec{\beta}}}(\sigma \cdot \tau)$ is still invariant.
\end{remark}

\section{Polynomial differential operators on Riemannian model spaces}\label{sec:PDOs}
Let $(M,g)$ be a Riemannian manifold with Levi--Civita connection $\nabla$. Throughout the paper, all differential operators will be constructed from covariant derivatives of functions by means of contractions with respect to the Riemannian metric.

\subsection{Iterated covariant derivatives}

Let us first consider the iterative covariant derivative of a function.
Assume that $f \in C^k(M)$ with differential $df$. For vector fields $Y_0, \dots, Y_{j-1}$, let 
\begin{align*}
    \nabla_{Y_0} f &= df(Y_0)\\
    \nabla_{Y_0,Y_1}^2 f &= \nabla_{Y_0} \nabla_{Y_1} f - \nabla_{\nabla_{Y_0} Y_1} f,\\
    \nabla_{Y_0,\dots,Y_j}^{r+1} f &= \nabla_{Y_0} \nabla_{Y_1, \dots, Y_j}^r f - \nabla_{\nabla_{Y_0} Y_1, \dots, Y_j}^j f - \cdots - \nabla_{Y_1, \dots, \nabla_{Y_0}Y_r}^j f,
\end{align*}
for $2\leq j$.
We can introduce $r$-tensors $\nabla^rf$ for $0\leq r$ by definition through iterated covariant derivatives as $\nabla^r f(Y_1, \dots, Y_r):=\nabla^r_{Y_1, \dots, Y_r} f $, and have the convention that $\nabla^1_{Y_0} = \nabla_{Y_0}$ and $\nabla^0 = \id$. Observe that
$$\nabla^2_{X,Y}f - \nabla_{Y,X}^2f = R(X,Y)f=0, \qquad X,Y \in \Gamma(TM)$$
with $R$ the curvature operator. As $R(X,Y)f=0$ for any function $f \in C^\infty(M)$, we deduce that $\nabla^2_{X,Y} f = \nabla^2_{Y,X}f$, i.e. the Hessian of a function is symmetric.
On the other hand, higher derivatives are not symmetric on non-flat spaces, as shown by a short calculation
\[\left(\nabla^3_{Y_1, Y_2, Y_3} -\nabla^3_{Y_2,Y_1, Y_3}\right)f =\nabla_{R(Y_1, Y_2)Y_3}f\]
which vanishes only when the curvature vanishes. Hence, $\nabla^r f$ is not symmetric whenever the curvature is non-vanishing.
However, the skew-symmetric parts can be expressed as derivatives of lower order, so all the information is captured by the symmetrized $r$-th derivatives $\nabla^{r,\Sym} f$, 
\[
\nabla_{Y_{1}, \dots, Y_k}^{r,\Sym}f = \frac{1}{r!} \sum_{\sigma \in S_r} \nabla_{Y_{\sigma(1)}, \dots, Y_{\sigma(r)}}^r f
\]
with $S_r$ being the permutation group of order $r$. With $\circlearrowright$ denoting the cyclic sum we remark, for example, that $\nabla^{2,\Sym}_{Y_1,Y_2}f = \nabla^{2}_{Y_1,Y_2}f$,
\begin{align}
\nabla^{3,\Sym}_{Y_1,Y_2,Y_3}f & = \frac{1}{3} \circlearrowright \nabla^{3}_{Y_1,Y_2,Y_3} f \label{3Sym} \\ \nonumber
& = \nabla_{Y_1, Y_2, Y_3}^3 f + \frac{1}{3} \nabla_{R(Y_2,Y_1)Y_3} f + \frac{1}{3} \nabla_{R(Y_3, Y_1)Y_2}f.
\end{align}
Let now $\varphi$ be any isometry of $(M,g)$ and $Y$ a vector field on $M$. We define $\varphi_\ast  Y$ as the vector field $x \mapsto \varphi_\ast  (Y(\varphi^{-1}(x)))$. Recall the property of the Levi-Civita connection $\nabla$, we have $\varphi_\ast  \nabla_X Y = \nabla_{\varphi_\ast  X} \varphi_\ast  Y$, $X,Y \in \Gamma(TM)$ (see, e.g., \cite[Prop. 5.6]{lee2018introduction}) which when applied iteratively produces

\[
\nabla^r (f \circ \varphi)(Y_1, \dots, Y_k) = \nabla^r f(\varphi_\ast  Y_1, \dots, \varphi_\ast  Y_r) \circ \varphi,
\]
\begin{equation} \label{EquaNabla} 
\nabla^{r,\Sym} (f \circ \varphi)(Y_1, \dots, Y_k) = \nabla^{r,\Sym} f(\varphi_\ast  Y_1, \dots, \varphi_\ast  Y_k) \circ \varphi.
\end{equation}
If $X \in \Gamma(TM^{\otimes r})$ is a multi-vector field with $\varphi_*X: x\mapsto \varphi_* (X(\varphi^{-1}(x)))$, we can write the above equalities as
\begin{equation} \label{Xvarphistar}
\nabla^r(f \circ \varphi)(X)  = \nabla^rf(\varphi_*X) \circ \varphi.
\end{equation}
We also observe the identity $\nabla^{r+1}_{Y \otimes X} f = \nabla_Y \nabla_X^r f - \nabla_{\nabla_Y X}^r f$ for vector field $Y \in \Gamma(TM)$ and multi-vector field $X \in \Gamma(TM^{\otimes r})$.

\begin{remark}
Since we will need $\nabla^r f$, $f \in C^\infty$ to denote an $r$-tensor, this will mean that $\nabla^1 f = \nabla f = df$ is then a 1-tensor. However, in other expression, in particular, when writing out examples of equivariant differential operators, it will still be convenient to let $\nabla f = \sharp df$ be a vector field. We will allow for this ambiguity in notation, letting the context determine if $\nabla f$ denotes a form or a vector field.
\end{remark}

\subsection{Nonlinear polynomial differential operators} \label{sec:DiffOp}

The goal of this section is to make precise the class of operators studied throughout the paper and to show, in Proposition~\ref{prop:SpanNPDO}, that every such operator can be expressed as a finite linear combination of products of symmetric covariant derivative operators, which is a reduction that we will take advantage of in subsequent sections.

Let $L_1, \dots, L_l$ be linear differential operators on a smooth manifold $M$. For a polynomial $p$ in $l$ variables, the operator
\[
    Pf := p(L_1 f, \dots, L_l f)
\]
is called a \emph{nonlinear polynomial differential operator of order $\leq k$}.
We write $\NPDO(M)$ for the linear span of all operators of this form. The \emph{polynomial degree} of $P$ is the degree of $p$; when $p$ is homogeneous of degree $n$, equivalently $P(cf) = c^n Pf$ for all $c \in \mathbb{R}$, we say that $P$ has
\emph{homogeneous polynomial degree $n$}.

We introduce a second notion of degree that accounts for the differential order of each factor. When $p$ is a monomial and $Pf = (L_1 f)^{s_1} \cdots (L_l f)^{s_l}$, with each $L_j$ having differential order $r_j$, the \emph{total order} of $P$ is $\sum_{j=1}^l s_j r_j$. For a general $P = \sum_{r=1}^{r_{\max}} P_r$, a finite sum of monomials, the total order of $P$ is the
maximum of the total orders of the $P_r$.

We remark that $\NPDO(M)$ depends only on the smooth structure of $M$: the operators $L_j$ range over all linear differential operators of order $\leq k$, with no reference to any metric or connection. A Riemannian metric contributes to the definition only when singling out specific elements of $\NPDO(M)$, such as $\Delta f = \tr \nabla^2_{*,*}f$ or $\|\nabla f\|^2$.

Consider now a Riemannian metric $g$ on $M$ with Levi-Civita connection $\nabla$. For any $0 \leq r \leq k$ and $X \in \Gamma(\Sym^r TM)$, we let it act as a differential operator on function by
\begin{equation} \label{XSym} Xf := \nabla^{r,\Sym} f(X).\end{equation}
If $L$ is a linear differential operator of order $r$, we say that $X \in \Gamma(\Sym^r TM)$ is \emph{the leading symmetric part of $L$} if $L - X$ has order strictly less than $r$. Every linear differential operator of order $r$ has a unique leading symmetric part with respect to $\nabla$.

\begin{remark}
Uniqueness is immediate: if $X, X' \in \Gamma(\Sym^r TM)$ both satisfy the condition, then $X - X'$ acts by $f \mapsto \nabla^{r,\Sym}f(X - X')$ and has order strictly less than $r$, which forces $X = X'$ since $\nabla^{r,\Sym}f$ can realize any symmetric $r$-tensor at any point (take $f$ to be a degree-$r$ monomial in normal coordinates). For existence, work in local coordinates: write $L = \sum_{|\alpha| \leq r} a^\alpha \partial_\alpha$ and define $X$ via the symmetrized leading coefficients $X^{i_1 \cdots i_r} = a^{(i_1 \cdots i_r)}$. Under a change of coordinates, the Christoffel symbols of $\nabla$ contribute only lower-order terms to $\nabla^{r,\Sym}$, so the difference $L - X$ remains of order strictly less than $r$ in every chart, confirming that $X$ is a globally well-defined section of $\Sym^r TM$.
\end{remark}

\begin{proposition} \label{prop:SpanNPDO}
The space $\NPDO^p(M)$ is spanned by operators of the form
\[
    Pf = (X_1 f)(X_2 f) \cdots (X_n f), \qquad n \geq 0,
\]
where each $X_i \in \Gamma(\Sym^{r_i} TM)$ for some , $r_i \geq 0$, $\sum_{i=1}^n r_i \leq p$.
\end{proposition}

\begin{proof}
We proceed by induction on the total order of $P$.

When the total order is $0$, the operator $Pf$ reduces to a monomial in $f$, which already has the stated form with all $r_i = 0$.

For the inductive step, let $Pf = \prod_{j=1}^l (L_j f)^{s_j}$ be a monomial of total order $s > 0$.
Let $X_j \in \Gamma(\Sym^{r_j} TM)$ be the leading symmetric part of $L_j$, and write $L_j f = X_j f + R_j f$ where $R_j$ plays the role of the remainder and has order strictly less than $r_j$. Substituting and expanding,
\[
    Pf = \prod_{j=1}^l (X_j f)^{s_j} + Qf,
\]
where $Q$ is a linear combination of monomials, each containing at least one factor of $R_j f$ for some $j$.
Any such monomial has total order strictly less than $s$, so by the induction hypothesis $Q$ is already in the span of operators of the desired form. Since $\prod_{j=1}^l (X_j f)^{s_j}$ also has the desired form, the result follows.
\end{proof}


\subsection{Traces and contractions}

The products $(X_1 f)\cdots(X_n f)$ introduced in the previous section are scalar-valued, but they are built from covariant derivatives of $f$, which are covariant tensors. In order to construct scalar-valued operators from those tensors directly we need a way to contract tensor indices against each other, which comes natural by using the Riemannian metric. This section sets up such machinery precisely.

Let $(M,g)$ be a Riemannian manifold. The metric $g$ provides natural isomorphisms between the tangent and cotangent bundles:
\[
    \flat : TM \to T^\ast M, \quad (\flat v)(w) = g(v,w),
    \qquad
    \sharp : T^\ast M \to TM, \quad \text{the inverse of } \flat.
\]
Both extend to tensor products and we use the same symbols $\flat$ and $\sharp$ for the induced isomorphisms on $T^\ast M^{\otimes r}$ and $TM^{\otimes r}$.

The \emph{Riemannian cometric} $g^\ast  \in \Gamma(\Sym^2 TM)$ is the contravariant metric tensor defined by
\[
    g^\ast (\alpha,\beta) := \langle \sharp\alpha,\, \sharp\beta \rangle_g,
    \qquad \alpha,\beta \in T^\ast M.
\]
Since $g^\ast $ takes covectors as inputs, it is a section of $\Sym^2 TM$. In a local coordinate system $(x^1,\dots,x^d)$ with $g_{ij} = \langle \partial_{x^i}, \partial_{x^j}\rangle_g$, the  cometric has components $g^{ij} = \langle dx^i, dx^j\rangle_{g^\ast }$, which are the entries of the inverse matrix $(g_{ij})^{-1}$.

Let $T \in \Gamma(T^\ast M^{\otimes 2})$. The \emph{metric trace} of $T$ is the scalar function obtained by contracting both arguments against the cometric:
\[
    \tr T \;:=\; \sum_{i,j=1}^d g^{ij}\, T(\partial_{x^i}, \partial_{x^j}).
\]
We will use the asterisk-label notation such as illustrated below for these traces
\begin{equation} \label{TrEx}
\begin{aligned}
 \Delta f = \tr \nabla^2 f = \tr \nabla^2_{\ast ,\ast } f. \\
    \|\nabla f\|^2
    = \langle df,\, df \rangle_{g^\ast }
    = \tr\bigl(\nabla_{\ast } f\bigr)\bigl(\nabla_{\ast } f\bigr),\\
    \nabla^2 f(\nabla f,\nabla f)
    = \tr\bigl(\nabla^2_{\ast _1,\ast _2} f\bigr)
      \bigl(\nabla_{\ast _1} f\bigr)
      \bigl(\nabla_{\ast _2} f\bigr)
\end{aligned}
\end{equation}
%
Iterating, one may perform up to $\lfloor r/2 \rfloor$ contractions of an $r$-tensor. When $r$ is even and one contracts all $r/2$ pairs, no free indices remain and the result is a smooth function on $M$.

\section{Differential operators from multigraphs and examples} \label{sec:LandtildeL}
\subsection{From multigraphs to nonlinear differential operators}
We now have all the ingredients to define the central object of this paper:
a scalar-valued nonlinear differential operator $N_\gamma \in \NPDO(M)$
associated to each equivalence class $[\gamma] \in \MG$.
The construction is guided by two simple rules:
each vertex $v \in V[\gamma]$ contributes the symmetrized covariant derivative
$\nabla^{\deg(v),\Sym} f$,
and each edge $e \in E[\gamma]$ contracts a pair of tensor indices via the
Riemannian metric $g$.
We illustrate this with a concrete example before giving the general definition.

\begin{example} \label{ex:PathGraph}
Consider the path multigraph $\gamma$ with three vertices and two edges,
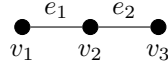
\begin{figure}[H]
\centering
\begin{tikzpicture}[scale=0.9]
  \node[circle, draw, fill=black, inner sep=2pt, label=below:$v_1$] (V1) at (0,0) {};
  \node[circle, draw, fill=black, inner sep=2pt, label=below:$v_2$] (V2) at (1,0) {};
  \node[circle, draw, fill=black, inner sep=2pt, label=below:$v_3$] (V3) at (2,0) {};
  \draw (V1) -- (V2) node[midway, above] {$e_1$};
  \draw (V2) -- (V3) node[midway, above] {$e_2$};
\end{tikzpicture}
\end{figure}
\noindent
where $\deg(v_1)=1$, $\deg(v_2)=2$, $\deg(v_3)=1$.
To each vertex we associate a symmetrized covariant derivative of $f$:
\[
  \alpha^{v_1} = \nabla^{1,\Sym} f = \nabla f,
  \qquad
  \alpha^{v_2} = \nabla^{2,\Sym} f = \nabla^2 f,
  \qquad
  \alpha^{v_3} = \nabla^{1,\Sym} f = \nabla f.
\]
Forming their tensor product gives a covariant $4$-tensor:
\[
  \nabla^{V[\gamma]} f
  \;:=\;
  \alpha^{v_1} \otimes \alpha^{v_2} \otimes \alpha^{v_3}
  \;=\;
  \nabla f \otimes \nabla^2 f \otimes \nabla f
  \;\in\; \Gamma\!\left(T^\ast M^{\otimes 4}\right).
\]
Each edge prescribes one metric contraction between the tensors at its endpoints, with $e_1$ contracting the single slot of $\alpha^{v_1}=\nabla f$ with one slot of $\alpha^{v_2}=\nabla^2 f$ using the cometric $g^\ast $, while edge $e_2$ contracts the remaining slot of $\alpha^{v_2}=\nabla^2 f$ with the single slot of $\alpha^{v_3}=\nabla f$. Since $\nabla^2 f$ is symmetric, the order in which we apply the two contractions does not matter.
After both contractions no free indices remain, and we obtain the scalar function
\[
  N_\gamma f
  \;=\;
  \tr\bigl(\nabla_{\ast_1} f\bigr)\bigl(\nabla^2_{\ast_1,\ast_2} f\bigr)\bigl(\nabla_{\ast_2} f\bigr)
  \;=\;
  \nabla^2 f(\nabla f, \nabla f).
\]
\end{example}

Let $(M,g)$ be a Riemannian manifold with Levi-Civita connection $\nabla$, and let $\gamma$ be a multigraph. Fix any ordering $V[\gamma] = \{v_1, \dots, v_n\}$ of the vertices and write $l_i = \deg(v_i)$. We present now the general procedure for constructing a differential operator associated to a multigraph. For $f \in C^\infty(M)$, associate to each vertex $v_i$ the symmetrized covariant derivative $\nabla^{l_i, \Sym} f \in \Gamma(T^\ast M^{\otimes l_i})$, and form their tensor product
\[
  \nabla^{V[\gamma]} f
  \;:=\;
  \nabla^{l_1,\Sym} f \otimes \cdots \otimes \nabla^{l_n,\Sym} f
  \;\in\; \Gamma\!\left(T^\ast M^{\otimes 2p}\right),
\]
where $p = |E[\gamma]|$ and $\sum_{i=1}^n l_i = 2p$. Recall that if $v$ is an isolated vertex, then $\nabla^{v} f = \nabla^{0}f =f$.
Each edge $e \in E[\gamma]$ specifies a metric contraction between two slots of $\nabla^{V[\gamma]} f$. We distinguish two cases.
\begin{enumerate}[$\bullet$]
  \item \emph{Non-loop edge} $e = \{v_i, v_j\}$ with $v_i \neq v_j$:
        contract one available slot of $\nabla^{l_i,\Sym} f$ against one
        available slot of $\nabla^{l_j,\Sym} f$ using the cometric $g^\ast $.
  \item \emph{Loop} $e = \{v_i, v_i\}$: contract two available slots
        of $\nabla^{l_i,\Sym} f$ against each other using $g^\ast $.
        This is well-defined because $\nabla^{l_i,\Sym} f$ is fully symmetric,
        so the choice of which two slots to contract is immaterial.
\end{enumerate}
Since each $\nabla^{l_i,\Sym} f$ is symmetric, the contractions along different edges commute, so we may write
\[
  N_\gamma = \tr_{E[\gamma]} \nabla^{V[\gamma]} f
  \;:=\;
  \tr_{e_1} \circ \cdots \circ \tr_{e_p} \,\nabla^{V[\gamma]} f
\]
without ambiguity for any ordering $E[\gamma] = \{e_1, \dots, e_p\}$.
After all $p$ contractions the result is scalar-valued.


\begin{remark}
The total number of slots in $\nabla^{V[\gamma]} f$ equals
$\sum_{i=1}^n l_i = 2p$, and each of the $p$ edge-contractions removes two
slots, leaving a scalar. In particular, $N_\gamma$ is well-defined for every
multigraph with vertex degrees bounded by $k$, including those with loops and
multiple edges.
\end{remark}

\begin{proposition} \label{prop:NWellDefined}
The operator $N_\gamma f = \tr_{E[\gamma]} \nabla^{V[\gamma]} f$ satisfies the
following.
\begin{enumerate}[\rm (i)]
  \item It does not depend on the ordering chosen for $E[\gamma]$.
  \item It does not depend on the ordering chosen for $V[\gamma]$.
  \item If $\Phi \colon \gamma \to \tilde\gamma$ is a multigraph isomorphism,
        then $N_\gamma = N_{\tilde\gamma}$.
        In particular, $N_\gamma$ depends only on the isomorphism class
        $[\gamma] \in \MG$.
\end{enumerate}
\end{proposition}

\begin{proof}
\textit{(i).}
Fix any ordering $V[\gamma] = \{v_1,\dots,v_n\}$. Since $\nabla^{l_i,\Sym} f$ is a fully symmetric tensor, contracting along two edges $e_1, e_2 \in E[\gamma]$ in either order acts on symmetric slots of $\nabla^{V[\gamma]} f$ and therefore commutes. By induction, the iterated contraction $\tr_{e_1} \circ \cdots \circ \tr_{e_p}$ yields the same scalar for every ordering of $E[\gamma]$. 

\textit{(ii).}
Let $V[\gamma] = \{v_1,\dots,v_n\}$ and $V[\gamma] = \{v_{\sigma(1)},\dots,v_{\sigma(n)}\}$ be two orderings of the nodes, related by a permutation $\sigma \in S_n$ which also induces a reordering of the edges. The two orderings produce tensor products that differ by a permutation of their factors:
\[
  \nabla^{l_{\sigma(1)},\Sym} f \otimes \cdots \otimes \nabla^{l_{\sigma(n)},\Sym} f
  \;=\;
  \pi_\sigma \cdot \bigl(\nabla^{l_1,\Sym} f \otimes \cdots \otimes \nabla^{l_n,\Sym} f\bigr),
\]
where $\pi_\sigma$ permutes the tensor factors in the corresponding blocks of slots. At the same time, each edge $e = \{v_i, v_j\}$ of the new ordering now contracts the slots belonging to $v_{\sigma^{-1}(i)}$ and $v_{\sigma^{-1}(j)}$, i.e., the contraction pattern is permuted by $\pi_\sigma$ in exactly the same way. Since a complete contraction of all $p$ edges is invariant under simultaneous permutation of both the tensor product and the contraction pattern, the result $\tr_{E[\gamma]} \nabla^{V[\gamma]} f$ is unchanged.

\textit{(iii).}
Let $\Phi \colon \gamma \to \tilde\gamma$ be a graph isomorphism. It suffices to remark that any graph isomorphism is a reordering of the vertices, which in turn induces a reordering of the edges, to conclude that by (ii), the result is independent of vertex ordering, so
\[
  N_{\tilde\gamma} f
  \;=\;
  \tr_{E[\tilde\gamma]} \tilde\nabla^{V[\tilde\gamma]} f
  \;=\;
  \tr_{E[\gamma]} \nabla^{V[\gamma]} f
  \;=\;
  N_\gamma f. \qedhere
\]
\end{proof}

%

\subsection{From parametrizations to nonlinear differential operators} Recall from Section~\ref{sec:Tensors} the $\Ort(d)$-invariant tensors $\tau_\rho \in \calT_d^{2p}$ on $\mathbb{R}^d$ defined for $\rho \in \frakR_p$. Their manifold analogue is defined as follows. For $\rho = \{\{i_1,i_2\},\dots,\{i_{2p-1},i_{2p}\}\} \in \frakR_p$, define $T_\rho \in \Gamma(T^\ast M^{\otimes 2p})$ by
\[
  T_\rho(\xi_1,\dots,\xi_{2p})
  \; :=\;
  \prod_{\{i,j\}\in\rho} \langle \xi_i, \xi_j \rangle_g,
  \qquad \xi_1,\dots,\xi_{2p} \in TM.
\]
\noindent In the above definition, we adopt the convention that $T_\emptyset =1$ for the case $p =0$. Now let $\beta \in \DV$ be a degree vector with $|\beta|_E = p$. Recall that the canonical map $m_{\beta} \colon \{1,\dots,2p\} \to V_{\beta}$ partitions the indices $\{1,\dots,2p\}$ into consecutive fibers, one per vertex, ordered by non-decreasing degree. Define the tensor
\begin{equation} \label{nablavecm}
  \nabla^{\beta} f
  \;:=\;
  \bigotimes_{j=0}^k \bigl(\nabla^{j,\Sym} f\bigr)^{\otimes \beta(j)}
  \;\in\; \Gamma\!\left(T^\ast M^{\otimes 2p}\right),
\end{equation}
where factors with $\beta(j)=0$ are omitted. If $\beta = 0$, then we adopt the convention that $\nabla^{\beta}f =1$. It will also be convenient to write $\nabla^{\vec{\beta}} := \nabla^{(0,\vec{\beta})}$.

\begin{definition}
For $\rho \in \frakR_p$, a degree vector $\beta = (\beta_0, \vec{\beta})$ with $|\beta|_E = p$, define
\[
  \tilde N(\rho,\beta)\,f = \tilde N(\rho, \beta_0, \vec{\beta}) \;:=\; \nabla^{\beta} f\,(\sharp T_\rho) = f^{\beta_0} \nabla^{\vec{\beta}} f\,(\sharp T_\rho),
\]
\end{definition}

The contravariant tensor $\sharp T_\rho$ acts as a \emph{contraction instruction}: for each pair ${\{i,j\}\in\rho}$ it contracts slot $i$ against slot $j$ using the cometric $g^\ast $, i.e.\ it inserts a factor $\langle \cdot\,, \cdot \rangle_{g^\ast }$ between the corresponding arguments. Since $\rho$ is a perfect matching on $\{1,\dots,2p\}$, every slot is contracted exactly once and the result is a scalar. 

As an example, for $\rho = \{\{1,3\},\{2,4\}\}$ and $\vec{\beta} = [2,1]$ (two degree-$1$ vertices and one degree-$2$ vertex), one has that $\nabla^{\vec{\beta}} f = \nabla f \otimes \nabla f \otimes \nabla^2 f$, and then
\[
  \tilde N\!\bigl(\{\{1,3\},\{2,4\}\},[2,1]\bigr)\,f
  \;=\;
  \nabla^2 f(\nabla f,\nabla f),
\]
recovering the operator from Example~\ref{ex:PathGraph}.

\subsection{Relation to multigraphs} The parametric formula and the graph-theoretic definition are related as follows.

\begin{proposition} \label{prop:NtildeWellDefined}
Let $\beta = (\beta_0, \vec{\beta})$ be a degree vector with $|\beta|_E = p$.
\begin{enumerate}[\rm (i)]
  \item For any $\sigma \in S_{\vec{\beta}}$, one has
        $\tilde N(\sigma\cdot\rho,\beta) = \tilde N(\rho,\beta)$.
  \item If $[\gamma] = \Gamma(\rho,\beta)$, then
        $\tilde N(\rho,\beta) = N_\gamma$.
\end{enumerate}
\end{proposition}

\begin{proof}
\textit{(i).}
Recall that $S_{\vec{\beta}}$ is generated by two subgroups, $S_{0,\vec{\beta}}$ and $S_{1,\vec{\beta}}$ (Section~\ref{sec:UnorderedPairs}), so it suffices to verify invariance for each generator separately. For $\sigma \in S_{0,\vec{\beta}}$, $\sigma$ permutes indices within the fiber $(m_{\beta})^{-1}(v_I)$ for each vertex $v_I$, leaving every other index fixed. Since $\nabla^{l_I,\Sym} f$ is fully symmetric, the value $\nabla^{l_I,\Sym} f(\xi_{a_1},\dots,\xi_{a_{l_I}})$ is unchanged when the arguments are permuted within the fiber. Therefore $\nabla^{\beta} f(\sharp T_{\sigma\cdot\rho}) = \nabla^{\beta} f(\sharp T_\rho)$, i.e., $\tilde N(\sigma\cdot\rho,\beta) = \tilde N(\rho,\beta)$. For $\sigma = \sigma_{I_1 I_2} \in S_{1,\vec{\beta}}$, $\sigma$ swaps the fibers of two vertices $v_{I_1}$ and $v_{I_2}$ of the same degree $l$. In $\nabla^{\beta} f$, both $v_{I_1}$ and $v_{I_2}$ contribute a factor of $\nabla^{l,\Sym} f$. Swapping the corresponding blocks of indices in $\rho$ is equivalent to swapping the two identical tensor factors $\nabla^{l,\Sym} f$, which leaves their tensor product unchanged. Hence $\tilde N(\sigma\cdot\rho,\beta) = \tilde N(\rho,\beta)$.

\textit{(ii).}
By definition, $\Gamma(\rho,\beta)$ is the isomorphism class of $\tilde\Gamma(\rho,\beta)$, whose non-isolated vertices are $V_{\beta} = \{v_1,\dots,v_n\}$ in canonical order. By construction of $m_{\beta}$, the fiber $m_{\beta}^{-1}(v_I)$ consists of the $l_I$ indices assigned to $v_I$, and $\rho$ pairs these indices into edges. Unfolding the definition of $\tr_{E[\gamma]}$, each pair $\{i,j\} \in \rho$ contracts the slot indexed by $i$ (belonging to $v_{I_i}$) against the slot indexed by $j$ (belonging to $v_{I_j}$) using $g^\ast $, which is exactly what $\nabla^{\beta} f(\sharp T_\rho)$ computes, giving $\tilde N(\rho,\beta) f =  N_{\tilde\Gamma(\rho,\beta)} f = N_\gamma f$.
\end{proof}

\begin{proposition} \label{prop:NEquivariant}
For any multigraph $\gamma \in \MG$ and any isometry $\varphi$ of $(M,g)$,
\[
  N_\gamma(f \circ \varphi) \;=\; (N_\gamma f) \circ \varphi
  \qquad \text{for all } f \in C^\infty(M).
\]
In other words, $N_\gamma \in \NPDO(M)^G$ for every $[\gamma] \in \MG$.
\end{proposition}

\begin{proof}
Write $V[\gamma] = \{v_1,\dots,v_n\}$ with $l_i = \deg(v_i)$ and let $Y_1,\dots,Y_n$ be arbitrary vector fields on $M$. By equation~\eqref{EquaNabla}, for each vertex $v_i$,
\[
  \nabla^{l_i,\Sym}(f \circ \varphi)(Y_1,\dots,Y_{l_i})
  \;=\;
  \bigl(\nabla^{l_i,\Sym} f\bigr)(\varphi_\ast  Y_1,\dots,\varphi_\ast  Y_{l_i}) \circ \varphi.
\]
Taking the tensor product over all vertices and applying all edge contractions, the pushforward $\varphi_\ast $ appears in each slot and is then immediately contracted against itself via the metric $g$.
Since $\varphi$ is an isometry, $g(\varphi_\ast  X, \varphi_\ast  Y) = g(X,Y)$, so each contraction is unaffected by $\varphi_\ast $. Therefore
\[
  N_\gamma(f \circ \varphi)
  \;=\;
  \tr_{E[\gamma]} \nabla^{V[\gamma]}(f \circ \varphi)
  \;=\;
  \bigl(\tr_{E[\gamma]} \nabla^{V[\gamma]} f\bigr) \circ \varphi
  \;=\;
  (N_\gamma f) \circ \varphi. \qedhere
\]
\end{proof}

\subsection{Further examples} \label{sec:Examples}

We collect here several concrete examples illustrating how multigraphs encode nonlinear differential operators. The key structural rule, already noted in the introduction, is that disjoint unions factor:
\[
  N_{\gamma_1 \cup \gamma_2} f \;=\; (N_{\gamma_1} f)(N_{\gamma_2} f),
\]
so it suffices to understand connected multigraphs. A complete list of $N_\gamma$ for all multigraphs with at most three edges is given in the appendix (Table~1).

\begin{example}[Isolated vertices and the trivial operators]
Let $\gamma = \bullet^{\beta_0}$ consist of $\beta_0$ isolated vertices and no edges. There are no contractions to perform and each isolated vertex contributes a factor of $f$, so
\[
  N_{\bullet^{\beta_0}} f \;=\; f^{\beta_0}.
\]
In particular the null graph gives the constant $N_{\bullet^0} f = 1$, and a single isolated vertex gives $N_{\bullet} f = f$.
\end{example}

\begin{example}[Single edge: squared gradient norm]
\label{ex:SingleEdge}
Let $\gamma$ be the multigraph with two vertices $v_1, v_2$ of degree $1$
connected by a single edge (
\begin{tikzpicture}[scale=0.3]
  \node[circle, draw, fill=black, inner sep=1.5pt] (A) at (0,0) {};
  \node[circle, draw, fill=black, inner sep=1.5pt] (B) at (2,0) {};
  \draw (A) to (B);
\end{tikzpicture}
). Then $\nabla^V f = \nabla f \otimes \nabla f$, and
the edge contracts both slots via the cometric $g^*$:

\[
  N_\gamma f \;=\; \langle \nabla f,\, \nabla f \rangle_g \;=\; \lVert\nabla f\rVert^2.
\]
\end{example}

\begin{example}[Single loop: Laplacian]
\label{ex:SingleLoop}
Let $\gamma$ be the multigraph with a single vertex $v$ of degree $2$ and a loop $e = \{v,v\}$ (
\begin{tikzpicture}[scale=0.3]
  \node[circle, draw, fill=black, inner sep=1.5pt] (A) at (0,0) {};
  \draw (A) to[out=60,in=120,loop] ();
\end{tikzpicture}
). Then $\nabla^V f = \nabla^{2,\Sym} f = \nabla^2 f$, and the loop contracts both slots of $\nabla^2 f$ against each other:
\[
  N_\gamma f \;=\; \tr \nabla^2 f \;=\; \Delta f.
\]
\end{example}

\begin{example}[Products of simpler operators]
By the disjoint-union rule, Examples~\ref{ex:SingleEdge} and~\ref{ex:SingleLoop} combine immediately. Taking $\gamma$ to be the disjoint union of the single-edge and single-loop multigraphs gives
\[
  N_\gamma f \;=\; \|\nabla f\|^2 \cdot \Delta f.
\]
More generally, adding $\beta_0$ isolated vertices yields
$N_\gamma f = f^{\beta_0} \cdot \|\nabla f\|^2 \cdot \Delta f$.
\end{example}

\begin{example}[Double edge]
\label{ex:DoubleEdge}
Let $\gamma$ have two vertices $v_1, v_2$ each of degree $2$, connected by two parallel edges (
\begin{tikzpicture}[scale=0.3]
  \node[circle, draw, fill=black, inner sep=1.5pt] (A) at (0,0) {};
  \node[circle, draw, fill=black, inner sep=1.5pt] (B) at (2,0) {};
  \draw (A) to[bend left=20] (B);
  \draw (A) to[bend right=20] (B);
\end{tikzpicture}
). Then $\nabla^V f = \nabla^2 f \otimes \nabla^2 f$, and the two edges contract the two slots of the first $\nabla^2 f$ against the two slots of the second:
\[
  N_\gamma f
  \;=\; \sum_{i,j,k,l} g^{ik} g^{jl}\,(\nabla^2 f)_{ij}\,(\nabla^2 f)_{kl}
  \;=\; \|\nabla^2 f\|^2.
\]
\end{example}

\begin{example}[Triangle]
\label{ex:Triangle}
Let $\gamma = K_3$ be the complete graph on three vertices $v_1, v_2, v_3$, each of degree $2$, with one edge between every pair. Then $\nabla^V f = \nabla^2 f \otimes \nabla^2 f \otimes \nabla^2 f$, and the three edges contract the slots in a cycle:
\[
  N_{K_3} f \;=\; \tr (\nabla^2 f_{*_1,*_2})(\nabla^2 f_{*_1,*_3})(\nabla^2 f_{*_2,*_3})
\]
\end{example}

\begin{example}[Loop with a pendant edge]
\label{ex:LoopEdge}
Let $\gamma$ be the multigraph with two vertices $v_0, v_1$, a loop $\{v_0,v_0\}$ and a single edge $\{v_0,v_1\}$
(\begin{tikzpicture}[scale=0.3, baseline=-2pt]
  \node[circle, draw, fill=black, inner sep=1.5pt] (A) at (0,0) {};
  \node[circle, draw, fill=black, inner sep=1.5pt] (B) at (2,0) {};
  \draw (A) to[out=60,in=120,loop] ();
  \draw (A) -- (B);
\end{tikzpicture}).
The vertex degrees are $\deg(v_0)=3$ and $\deg(v_1)=1$, so $\nabla^V f = \nabla^{3,\Sym}f \otimes \nabla f$. The loop contracts two of the three slots of $\nabla^{3,\Sym}f$, leaving a $1$-tensor; the edge then pairs that $1$-tensor with $\nabla f$:
\[
  N_\gamma f
  \;=\; \bigl\langle \operatorname{tr}\nabla^{3,\Sym}_{*,*, \cdot } f,\; \nabla f \bigr\rangle_{g^\ast} = \tr \nabla^{3,\Sym}_{*_1,*_1,*_2} f \nabla_{*_2}f.
\]
Using \eqref{3Sym}, we have
\begin{align*}
N_{K_3}f
& = \tr (\nabla_{*_1, *_2, *_2}^3 f)(\nabla_{*_1}f)  + \frac{2}{3} \tr (\nabla_{R(*_2, *_1)*_2}f)(\nabla_{*_1}f) \\
& = \langle \nabla \Delta f, \nabla f \rangle_{g^\ast} - \frac{2}{3} \operatorname{Ric}(\nabla f, \nabla f).
\end{align*}
For the case of $M = \mathbb{R}^d$ with an Euclidean metric, then $N_\gamma f = \langle \nabla \Delta f, \nabla f \rangle_g$, but we have the same equality on any Riemannian manifold up to operators of lower total order. On a Riemannian manifold of constant sectional curvature $K$, then
$$N_\gamma f = \langle \nabla \Delta f, \nabla f \rangle_{g^\ast} - \frac{2(d-1)K}{3} \|\nabla f \|^2,$$
where the last operator is in Example~\ref{ex:SingleEdge}. 
\end{example}

\begin{remark}
Up to lower-order terms, the contraction pattern of $N_\gamma$ is the same on any Riemannian manifold as on $\R^d$. For constant curvature spaces, these lower order terms can always be described by images of multigraphs by fewer edges by Theorem~\ref{th:main}. In particular, when listing operators in $\NPDO^p(M)$ up to a given total order $p$, it is sufficient to include terms of maximal total order for every multigraph.
\end{remark}



For multigraphs with many parallel edges or loops, the \emph{weighted-graph notation} is more compact: an integer node weight records the number of loops at a vertex and an integer edge weight records the multiplicity of the edge. 

\begin{figure}[h]
\centering
\begin{tikzpicture}[scale=0.7]

\begin{scope}[xshift=0cm]
  \node[circle, draw, fill=black, inner sep=2pt] (A1) at (0,0) {};
  \node[circle, draw, fill=black, inner sep=2pt] (B1) at (1,0) {};
  \node[circle, draw, fill=black, inner sep=2pt] (C1) at (2,0) {};

  \draw (A1) to[bend left=12] (B1);
  \draw (A1) to[bend right=12] (B1);
  \draw (B1) to (C1);
  \draw (A1) to[out=60,in=120,loop] ();

  \node at (1,-0.8) {(a)};
\end{scope}

\begin{scope}[xshift=5cm]
  \node[circle, draw, fill=black, inner sep=2pt, label=$2$] (A2) at (0,0) {};
  \node[circle, draw, fill=black, inner sep=2pt, label=$0$] (B2) at (1,0) {};
  \node[circle, draw, fill=black, inner sep=2pt, label=$0$] (C2) at (2,0) {};

  \draw (A2) -- (B2) node [midway, below] {$2$};
  \draw (B2) -- (C2) node [midway, below] {$1$};

  \node at (1,-0.8) {(b)};
\end{scope}

\end{tikzpicture}
\caption{Two alternative representations of the same differential operator $\tr (\nabla^{4,\Sym}_{\ast_0,\ast_0,\ast_1,\ast_2} f)(\nabla^{3,\Sym}_{\ast_1,\ast_2,\ast_3}f)(\nabla_{\ast_3} f )$ as a multigraph (a) and as a weighted graph~(b).}
\end{figure}
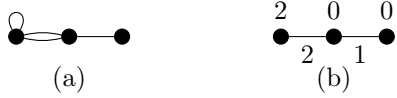

\begin{example}[Star multigraph]
Let $\gamma$ be the star multigraph with central vertex $v_4$ and three leaf vertices $v_1, v_2, v_3$. We allow $n_j \geq 0$ loops at vertex $v_j$ and $e_j \geq 1$ edges between $v_4$ and $v_j$, for $j=1,2,3$.
In the weighted-graph representation this is depicted as:
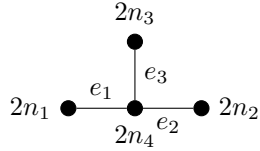
\begin{figure}[H]
    \centering
\begin{tikzpicture}[scale=1.1]
  \node[circle, draw, fill=black, inner sep=2pt, label=left:{$2n_1$}] (A) at (0,0) {};
  \node[circle, draw, fill=black, inner sep=2pt, label=below:{$2n_4$}] (B) at (0.8,0) {};
  \node[circle, draw, fill=black, inner sep=2pt, label=right:{$2n_2$}] (C) at (1.6,0) {};
  \node[circle, draw, fill=black, inner sep=2pt, label=above:{$2n_3$}] (D) at (0.8,0.8) {};

  \draw (A) -- (B) node [midway, above] {$e_1$};
  \draw (B) -- (C) node [midway, below] {$e_2$};
  \draw (B) -- (D) node [midway, right] {$e_3$};
\end{tikzpicture}
    \caption{Weighted star with four vertices. Node labels $n_j \geq 0$ denote the number of loops at vertex $v_j$; edge labels $e_j \geq 1$ denote the number of edges between the center $v_4$ and leaf $v_j$.}
\end{figure}

The degree of the central vertex is $\deg(v_4) = 2n_4 + e_1 + e_2 + e_3$, and the degree of leaf $v_j$ is $\deg(v_j) = 2n_j + e_j$, for $j = 1,2,3$. We can contract the corresponding indices to obtain the operator
\[
N_\gamma f = \left\langle (\operatorname{tr}_{*,*})^{n_4}\nabla^{2n_4+e_1+e_2+e_3,\Sym}f,\; \bigotimes_{j=1}^3 (\operatorname{tr}_{*,*})^{n_j}\nabla^{2n_j+e_j,\Sym}f \right\rangle_{g^\ast}
\]
which simplifies in the flat case as
\[
N_\gamma f = \left\langle \nabla^{e_1+e_2+e_3,\Sym}\Delta^{n_4}f,\; \bigotimes_{j=1}^3 \nabla^{e_j,\Sym}\Delta^{n_j}f \right\rangle_{g^\ast}
\]
As the symmetry group is $S_3$, which corresponds to changing the ordering of the leaves, we can render the encoding unique by imposing the condition $n_1\geq n_2 \geq n_3$ and if $n_1=n_2$ then $e_1\geq e_2$ and if  $n_2=n_3$ then $e_2\geq e_3$. In this way the representative within the symmetry group is uniquely determined.
\end{example}

\section{Frame bundles, associated functions and equivariance}\label{sec:FrameBundles}
\subsection{Nonlinear operators as sections} 
Let $\beta: \mathbb{Z}_{\geq 0} \to \mathbb{Z}_{\geq 0}$ be a function with only finitely many values being nonzero. We call such functions \emph{pseudo-degree vectors}, since we do not require that $|\vec{\beta}|_E$ is an integer. We will use the same notation otherwise for pseudo-degree vectors as for degree vectors. Write $\hDV$ for the collection of all pseudo-degree vectors. For any function $f \in C^\infty$, define $\nabla^{\beta}f$ as in \eqref{nablavecm}.

For any $\beta = (\beta_0, \vec{\beta}) \in \hDV$ with $|\beta|_{E} = \frac{p}{2}$, define $S_{\beta}\subset S_p$ be the subset of permutation defined as for degree vectors in Section~\ref{sec:UnorderedPairs}. If we define the action of $S_p$ on $TM^{\otimes p}$ by $\sigma \cdot (\xi_1 \otimes \cdots \otimes \xi_p) = \xi_{\sigma(1)} \otimes \cdots \otimes \xi_{\sigma(p)}$, we then let $\Sym^{\vec{\beta}} TM$ denote the subbundle of elements that are invariant under the action of $S_{\vec{\beta}}$. Define $\Sym^{\vec{\beta}} T^*M$ analogously.
\begin{lemma} \label{lemma:Pm0m}
For any $P \in \NPDO$, and for all $\beta = (\beta_0, \vec{\beta}) \in \hDV$, there is a unique section $P^\beta \in \Gamma(\Sym^{\vec{\beta}}  TM)$ such that
$$Pf = \sum_{\beta \in \hDV} \nabla^{\beta} f(P^{\beta}), \qquad f \in C^\infty(M),$$
with only finitely elements in the sum above being non-zero.
\end{lemma}
\begin{proof}
We use Proposition~\ref{prop:SpanNPDO}. Let us write $\NPDO(M) = \oplus_{\beta_0, \vec{\beta}}\NPDO^{\beta_0,\vec{\beta}}(M)$, where the subspace $\NPDO^{\beta_0,\vec{\beta}}(M)$ is spanned by operators
\begin{equation} \label{basis1} f  \mapsto f^{\beta_0} \cdot (X_1f)(X_2f) \cdots (X_p f), \qquad |\vec{\beta}|_E = \frac{p}{2},\end{equation}
where $X_j$ is a section of $\Sym^{r_j} TM$, and exactly $\vec{\beta}(k)$ of them equal $k$. If we label $X_1, \dots, X_n$ in non-decreasing order, then
\begin{equation} \label{basis2} f \mapsto f^{\beta_0} \cdot  \nabla^{\vec{\beta}}f (X_1 \otimes X_2 \otimes \cdots \otimes X_p) =  \nabla^{\beta}f (X_1 \otimes X_2 \otimes \cdots \otimes X_p).\end{equation}
Further symmetrizing $X_1 \otimes \cdots \otimes X_p$ over $S_{\vec{\beta}}$, we have a section of $\Sym^{\vec{\beta}} TM$ with the same values when evaluated by $\nabla^{\vec{\beta}} f$. Uniqueness finally follows from the fact that any element in $\Sym^{\vec{\beta}} T^*_xM$, $x\in M$ can be written as $\nabla^{\vec{\beta}} f(x)$ for some $f\in C^\infty(M)$.
\end{proof}

\subsection{Structure of the frame bundle}
We refer to, e.g., \cite{MR4722040} for details on the construction and relations on the orthogonal frame bundle. Let $(M,g)$ be a Riemannian manifold of dimension~$d$. For each $x \in M$, let $\Ort(M)_x$ denote the space of all linear isometries $u:\mathbb{R}^d \to T_xM$ from $\mathbb{R}^d$ with Euclidean structure to the tangent space $T_xM$. We have a transitive action of $\Ort(d)$ on $\Ort(M)_x$ by composition on the right. The group $\Ort(d)$ acts on $\Ort(M)_x$ freely and transitively by right composition, $u \cdot a := u \circ a$. Setting $\Ort(M) := \bigsqcup_{x\in M}\Ort(M)_x$, we obtain the principal bundle
\[
  \Ort(d) \;\longrightarrow\; \Ort(M) \;\stackrel{\pi}{\longrightarrow}\; M,
\]
called \emph{the orthonormal frame bundle} of $(M,g)$.
The tangent bundle $T\Ort(M)$ is spanned by two families of vector fields, defined as follows.
\begin{enumerate}[$\bullet$]
    \item \emph{Vertical fields.}
    For any $Z \in \so(d)$, define 
    \[\xi_Z(u) = \frac{d}{dt} u \circ e^{tZ} |_{t=0}\qquad u \in \Ort(M).\]
    Since $\pi(u \circ e^{tZ}) = \pi(u)$ for all $t$, we have $\pi_\ast \xi_Z = 0$, so the vertical field $\xi_Z$ is tangent to the fibers. The span $\mathcal{V}_u := \{\xi_Z(u) : Z \in \so(d)\}$ is a subspace of $T_u\Ort(M)$ of dimension $\dim\so(d) = \tfrac{d(d-1)}{2}$, called the \emph{vertical subspace}.

    \item \emph{Horizontal fields.}
    For $a \in \mathbb{R}^d$ and $u \in \Ort(M)$, let $c\colon(-\varepsilon,\varepsilon)\to M$ be any smooth curve with $c(0) = \pi(u)$ and $\dot{c}(0) = u(a)$, and let $\tilde{c}$ be its unique horizontal lift to $\Ort(M)$ with $\tilde{c}(0) = u$, i.e., $\tilde c$ is the result of parallel translating $u$ along $c$. Define $H_a(u) := \dot{\tilde{c}}(0)$.
    Since parallel transport depends linearly on the initial velocity, the map $a \mapsto H_a(u)$ is linear in $a$, so $\mathcal{H}_u := \{H_a(u) : a \in \mathbb{R}^d\}$ is a $d$-dimensional subspace of $T_u\Ort(M)$, called the \emph{horizontal subspace}.
    By construction, $\pi_\ast H_a(u) = u(a) \in T_{\pi(u)}M$, so $\pi_\ast|_{\mathcal{H}_u} \colon \mathcal{H}_u \to T_{\pi(u)}M$ is a linear isomorphism. We write $\mathcal{H} \subseteq T\Ort(M)$ for the rank-$d$ subbundle $\bigsqcup_{u}\mathcal{H}_u$.
\end{enumerate}

Since $\dim\mathcal{H}_u + \dim\mathcal{V}_u = d + \tfrac{d(d-1)}{2} = \tfrac{d(d+1)}{2} = \dim\Ort(M)$ and $\mathcal{H}_u \cap \mathcal{V}_u = \{0\}$, we obtain a direct-sum decomposition
\[
  T\Ort(M) \;=\; \mathcal{H} \oplus \mathcal{V}.
\]
The vertical subbundle $\mathcal{V}$ is intrinsic to the principal bundle structure, while $\mathcal{H}$ depends on the choice of connection $\nabla$, which in our case is Levi-Civita.


\subsection{Invariance under isometries}

Let now $G$ denote the isometry group of $(M,g)$. The action of $G$ on $M$ lifts to a left action on $\Ort(M)$ defined 
\[
  (\varphi \cdot u)(a) \;:=\; \varphi_\ast\, u(a),
  \qquad \varphi \in G,\; u \in \Ort(M),\; a \in \mathbb{R}^d.
\]
We write $\ell_\varphi \colon \Ort(M) \to \Ort(M)$ for the corresponding diffeomorphism $u \mapsto \ell_\varphi(u) =\varphi \cdot u$, which induces an action on $T\Ort(M)$ by $\varphi \cdot v := (\ell_\varphi)_\ast v$. Furthermore, isometries preserve horizontal and vertical subbundles by the property $\varphi_* \nabla_{X} Y = \nabla_{\varphi_* X} \varphi_* Y$.

For vertical fields, since the fiber over $\pi(u)$ is mapped to the fiber over $\varphi(\pi(u))$ and the right $\Ort(d)$-action commutes with $\ell_\varphi$, differentiating $\ell_\varphi(u \circ e^{tZ}) = (\varphi \cdot u) \circ e^{tZ}$ at $t = 0$ gives $\varphi \cdot \xi_Z(u) = \xi_Z(\varphi \cdot u)$. For horizontal fields, since $\ell_\varphi(\tilde{c})$ is the horizontal lift of $\varphi \circ c$ and $\dot{(\varphi \circ c)}(0) = \varphi_\ast u(a) = (\varphi \cdot u)(a)$, differentiating at $t = 0$ gives $\varphi \cdot H_a(u) = H_a(\varphi \cdot u)$.
%
These identities describe how the $G$-action interacts with $H_a$ and $\xi_Z$ pointwise. For the analysis of differential operators, we need to study the interaction with smooth functions. Given a smooth function $\mathbf{Y}:\Ort(M)\to\mathbb{R}^d$, we define the vector fields
\[
  H_{\mathbf{Y}} \colon u \mapsto H_{\mathbf{Y}(u)}(u).
\]
The following lemma records how $H_{\mathbf{Y}}$ interacts with $F\mapsto F\circ\ell_\varphi$.

\begin{lemma}\label{lem:IsomHorDir}
Let $F: \Ort(M) \to \mathbb{R}$ and $\mathbf{Y}: \Ort(M) \to \mathbb{R}^d$ be two smooth functions, and let $\ell_\varphi$ be the left action of the isometry group $\varphi \in G$ on $\Ort(M)$ and $T\Ort(M)$. Write $H_{\mathbf{Y}}$ for the vector field $u \mapsto H_{\mathbf{Y}(u)}(u)$. Then for any $\varphi \in G$
$$H_{\mathbf{Y}}(F \circ \ell_\varphi) = (H_{\mathbf{Y} \circ \ell_{\varphi^{-1}}} F) \circ \ell_{\varphi}.$$
\end{lemma}
\begin{proof}
Follows from
\begin{align*}
    H_{\mathbf{Y}(u)}(u) (F \circ \ell_\varphi) &= (\varphi \cdot H_{\mathbf{Y}(u)}(u)) F = ( H_{\mathbf{Y}(u)}(\varphi \cdot u)) F \\
    & = ( H_{(\mathbf{Y} \circ \ell_{\varphi^{-1}})(\varphi u)}(\varphi \cdot u)) F =( H_{(\mathbf{Y} \circ \ell_{\varphi^{-1}})} F)(\ell_\varphi(u)) .
\end{align*}
\end{proof}

This lemma naturally extends by induction from $\R^d$ to $\Sym^r\R^d$ as follows.

\begin{corollary}\label{cor:IsomHorDirHigher}
Let $e_1, \dots, e_d$ be the standard basis of $\mathbb{R}^d$. Let $\mathbf{X}: \Ort(M) \to (\mathbb{R}^d)^{\otimes r}$ be a function defined on $\Ort(M)$, and write $H_{\mathbf{X}}^r$, such that if $\mathbf{X} = \sum_{j_1, \dots, j_r=1}^d \mathbf{X}^{j_1,\dots, j_r} e_{j_1} \otimes \cdots \otimes e_{j_r}$ then
\begin{equation} \label{HbfX}
  H_{\mathbf{X}}^r  := \sum_{j_1,\dots,j_r=1}^d\mathbf{X}^{j_1\cdots j_r}\, H_{e_{j_1}}\cdots H_{e_{j_r}}.
\end{equation}
Then for any isometry $\varphi \in G$ and smooth $F : \Ort(M) \to \mathbb{R}$,
\[
  H_{\mathbf{X}}^r(F \circ \ell_\varphi) = (H^r_{\mathbf{X} \circ \ell_{\varphi^{-1}}} F) \circ \ell_\varphi.
\]
\end{corollary}
\begin{proof}
By using Lemma~\ref{lem:IsomHorDir} and induction, it follows that for any fixed $j_1, \dots, j_r \in \{1,\dots,d\}$,
\[
  H_{e_{j_1}} \cdots H_{e_{j_r}}(F \circ \ell_\varphi)(u)
  = \bigl(H_{e_{j_1}} \cdots H_{e_{j_r}} F\bigr)(\ell_\varphi u).
\]
The result then follows by multiplying by
$\mathbf{X}(u)(e_{j_1},\dots,e_{j_r})
 = \mathbf{X}(\ell_{\varphi^{-1}}(\ell_\varphi u))(e_{j_1},\dots,e_{j_r})$
and summing over all indices.
\end{proof}

\subsection{Horizontal lifts and associated maps} \label{sec:Associated}
Since the map $\pi_{\ast ,u}|_{\calH_u}: \calH_u \to T_{\pi(u)} M$ is invertible, we can write its inverse as $h_u: T_{\pi(u)} M \to \calH_u$. We call this map \emph{the horizontal lift} to $u \in \Ort(M)$. This map induces a lift from $TM^{\otimes r} \to \calH_u^{\otimes r}$, $r \geq 0$, which we will, with slight abuse of notation, denote by $h_u$ as well. Finally, the map $H_a \mapsto a = u^{-1}(\pi_\ast  H_a)$ induces a linear bijection from $\calH_u$ to $\mathbb{R}^d$, which is then extended to identifying $\calH_u^{\otimes r}$ with $(\mathbb{R}^d)^{\otimes r}$.
For every section $X \in \Gamma(TM^{\otimes r})$, we define its \emph{associated function} $\mathbf{X}: \Ort(M) \to (\mathbb{R}^d)^{\otimes r}$ such that if $\mathbf{X} = \sum_{j_1, \dots, j_r=1}^d \mathbf{X}^{j_1,\dots, j_r} e_{j_1} \otimes \cdots \otimes e_{j_r}$, then
$$h_u X(\pi(u)) = \sum_{j_1, \dots, j_r=1}^d \mathbf{X}^{j_1,\dots, j_r}(u)  H_{e_{j_1}}(u) \otimes \cdots \otimes H_{e_{j_r}}(u), \qquad u \in \Ort(d).$$
By definition, we have $\mathbf{X}(u \cdot a) = a^{-1} \cdot \mathbf{X}(u)$ for any $a \in \Ort(d)$, where the action on the right side is the action on $\Ort(d)$ on $(\mathbb{R}^d)^{\otimes r}$ induced by the standard representation on $\mathbb{R}^d$. The result below is standard.
\begin{lemma} \label{lemma:nablarHX}
Let $X \in \Gamma(TM^{\otimes r})$ be a multi-vector field with associated function $\mathbf{X}: \Ort(M) \to (\mathbb{R}^d)^{\otimes r}$. Let $H_{\mathbf{X}}$ be as in \eqref{HbfX}, $r \geq 0$. Then
$$(\nabla^r_{X}f) \circ \pi = H_{\mathbf{X}}^r (f \circ \pi), \qquad f \in C^\infty(M).$$
\end{lemma}
\begin{proof}
Let $X$ and $Y$ be respectively an $r$-multi-vector field and a vector field, with associated function $\mathbf{X}$ and $\mathbf{Y}$. We observe that by the definition of $\mathbf{X}$ and of $H_{\mathbf{Y}}$, we have
$$H_{\mathbf{Y}} \mathbf{X} = {\boldsymbol \nabla}_{\mathbf{Y}} \mathbf{X},$$
that is, the function associated to $\nabla_Y X$. It follows that, 
$$H_{\mathbf{Y} \otimes \mathbf{X}}^{r+1} = H_{\mathbf{Y}} H_{\mathbf{X}}^r - H_{H_{\mathbf{Y}} \mathbf{X}}^r = H_{\mathbf{Y}} H^r_{\mathbf{X}} - H_{{\boldsymbol \nabla}_{\mathbf{Y}} \mathbf{X}}^r .$$
We can now use the above identity, the fact that $H_{\mathbf{Y}}(f \circ \pi) = \nabla_Y f \circ \pi$ and induction to get the result.
\end{proof}

We extend this result to nonlinear operators. First, we define $H^{r,\Sym}_{a_1 \otimes \cdots \otimes a_r}$, by
$$H^{r,\Sym}_{a_1 \otimes \cdots \otimes a_r} = \frac{1}{r!} \sum_{\sigma \in S_r} H_{a_{\sigma(1)}} \cdots  H_{a_{\sigma(r)}},$$
with the convention that $H^{0,\Sym}F =F$ for a smooth $F \in C^\infty(\Ort(M))$.
We note that for any $F \in \Ort(M)$, $u \in \Ort(M)$, we can see $H^{r,\Sym}F(u): a_1 \otimes \cdots \otimes a_r \mapsto (H^{r,\Sym}_{a_1 \otimes \cdots \otimes a_r}F)(u)$ as an element in $(\mathbb{R}^{d,*})^{\otimes r}$. For a pseudo-degree vector $\beta\in \hDV$ with $|\beta|_{E} = \frac{p}{2}$, define $H^{\beta}_{\mathbf{a}}$, $\mathbf{a} \in (\mathbb{R}^d)^{\otimes p}$, such that if $\beta$ has order $k$, then
$$(H_{\mathbf{a}}^{\beta}F)(u) = \otimes_{r=0}^k (H^{r,\Sym}F(u))^{\beta(r)}(\mathbf{a}).$$
By applying Corollary~\ref{cor:IsomHorDirHigher} and Lemma~\ref{lemma:nablarHX} to each factor in the $H^{\beta}$, we get the following result.
\begin{lemma} \label{lemma:PreTh11}
Let $\pi: \Ort(M) \to M$ be the canonical projection. Let $\beta\in \hDV$ be a pseudo-degree vector with $|\beta|_{E} = \frac{p}{2}$. For any smooth functions $F \in C^\infty(M)$ and $\mathbf{X}: \Ort(M) \to (\mathbb{R}^d)^{\otimes p}$, if $\varphi \in G$ is an isometry of $M$, then
$$H_{\mathbf{X}}^{\beta} (F \circ \ell_\varphi) = (H_{\mathbf{X}}^{\beta} F) \circ \ell_\varphi.$$
Furthermore, if $\mathbf{X}$ is the function associated to $X \in \Gamma(TM^{\otimes r})$, then for any $f \in C^\infty(M)$,
$$(\nabla^{\beta}_X f) \circ \pi = H_{\mathbf{X}}^{\beta} (f \circ \pi).$$
\end{lemma}

We will then end with the following lemma that will be used in the proof of Theorem~\ref{th:main}.
\begin{lemma} \label{lemma:InvariantBfP}
Let $\varphi: M \to M$ be an isometry and $P \in \NPDO(M)$ be a nonlinear operator with $P^{\beta}$ defined as in Lemma~\ref{lemma:Pm0m}. If $\mathbf{P}^{\beta}$ is the map associated to $P^{\beta}$, then
$$P(f \circ \varphi) = (Pf) \circ \varphi \qquad \text{for every } f\in C^\infty(M),$$
if and only if for any $\beta \in \hDV$, the associated map $\mathbf P^{\beta}$ satisfies
\[\mathbf{P}^{\beta}(\varphi^{-1} \cdot u) = \mathbf{P}^{\beta}(u)\qquad\text{for any } u \in \Ort(M).\]
\end{lemma}

\begin{proof}
Throughout the proof, $f \in C^\infty(M)$ will be an arbitrary smooth function on $M$. Let us first consider the case when $P \in \NPDO^\beta(M)$ for fixed $\beta = (\beta_0 ,\vec{\beta}) \in \hDV$, so that we can write $Pf =\nabla^{\beta}f(P^{\beta})$. We first see that for any $\varphi \in G$,
$$P(f \circ \varphi) \circ \varphi^{-1} = \nabla^{\beta} f (\varphi_* P^{\beta}).$$
using \eqref{Xvarphistar} on every factor. It follows that $\NPDO^{\beta}(M)$ is preserved under the action of $\varphi$. Furthermore, we see that $P$ is equivariant if and only if $\varphi_* P^{\beta} = P^{\beta}$. On the other hand we see that
\begin{align*}
& P(f \circ \varphi) \circ \varphi^{-1} \circ \pi =  \nabla_{P^\beta}^{\beta} (f \circ \varphi) \circ \pi \circ l_{\varphi^{-1}} \\
& =  H_{\mathbf{P}^{\beta}}^{\beta} (f \circ \pi \circ \ell_{\varphi} )  \circ l_{\varphi^{-1}} = H_{\mathbf{P}^{\beta} \circ \ell_{\varphi^{-1}}}^{\beta} (f \circ \pi ).
\end{align*}
It follows that $\mathbf{P}^{\beta} \circ \ell_{\varphi^{-1}}$ is the function associated to $\varphi_* P^{\beta}$. If follows that $\varphi_* P^{\beta} = P^{\beta}$ if and only if $\mathbf{P}^{\beta} \circ \ell_{\varphi^{-1}} = \mathbf{P}^{\beta}$, which gives the result for our operator $P \in \NPDO^{\beta}(M)$.

For a general $P \in \NPDO(M)$, since the subspaces $\NPDO^{\beta}(M) \subset \NPDO(M)$ are preserved under the action of $\varphi$, we have that $P$ is equivariant if and only if $\mathbf{P}^{\beta} \circ \ell_{\varphi^{-1}} = \mathbf{P}^{\beta}$ holds for all~$\beta \in \hDV$.
\end{proof}

\section{Description of equivariant operators}\label{sec:EquivariantOperators}
\subsection{Riemannian model spaces}
We will give some general details about constant curvature spaces in this section, and refer the reader to, e.g., \cite[p. 204-207]{lee2018introduction}, \cite[Chapter XI.6-7]{kobayashi1996foundations}, \cite[Chapter 6.4]{sharpe2000differential} for details. Let $(M,g)$ be a simply connected Riemannian manifold with isometry group $G$. Assume that for every pair $x,y \in M$ and every linear isomorphism $q: T_x M \to T_yM$, there exists an isomorphism $\varphi \in G$ such that $\varphi_{\ast ,x} = q$. It follows that the action of the isometry group $G$ on $\Ort(M)$ is transitive. The Riemannian manifold $M = M^{d,K}$ will then be determined by its dimension $d$ and its constant sectional curvature $K$. If we let $\mathbb{R}^d$, $\mathbb{S}^d$ and~$\mathbb{H}^d$ denote the Euclidean space, the sphere, and the hyperbolic space of dimension $d$ respectively then for $K >0$,
$$M^{d,0} \cong \mathbb{R}^d, \qquad M^{d,K} \cong \sqrt{K} \cdot \mathbb{S}^d, \qquad M^{d,-K} \cong \sqrt{-K} \cdot \mathbb{H}^d.$$
Here $\cong$ denotes that the spaces are isomorphic. In these cases, the isomorphism groups are respectively Lie group isomorphic to $\E(d)$, $\Ort(d+1)$ and $\Ort(d,1)$. Their curvature tensor is given by
$$R(X,Y)Z = K\left(\langle Y, Z \rangle X - \langle X, Z \rangle Y\right).$$
For these spaces, we write the main result, as stated in Theorem~\ref{th:main}.



\subsection{Proof of Theorem~\ref{th:main}} \label{sec:ProofMain}
By Proposition~\ref{prop:NEquivariant}, each $N_\gamma$ is $G$-equivariant, so the 
map $N$ is well-defined as a linear map into $\NPDO(M)^G$. It remains to prove 
surjectivity, and bijectivity when restricted to $\spn_\mathbb{R}\MG^p$ with $p \leq d$.

For the rest of the proof $f \in C^\infty(M)$ will be an arbitrary smooth function. Conversely, assume that $P \in \NPDO(M)^G$. Considering the decomposition $\NPDO(M) = \oplus_{\beta \in \hDV} \NPDO^{\beta}(M)$, it is sufficient to consider $P \in \NPDO^{\beta}(M)$ for a fixed $\beta_0 \geq 0$ and $\beta \in \hDV$. Write $|\beta|_E = \frac{k}{2}$.

By our assumption, $Pf = \nabla^{\beta} f(P^{\beta})$ with associated function $\mathbf{P}^{\beta}$ which must satisfy $\mathbf{P}^{\beta}(\varphi \cdot u) = \mathbf{P}^{\beta}(u)$ for any $\varphi \in G$. However, since the action of $G$ on $\Ort(M)$ is transitive, it follows that $\mathbf{P}^{\beta}(\varphi \cdot u) = \mathbf{P}_0 \in (\mathbb{R}^d)^{\otimes k}$ is a constant. However, since $\mathbf{P}^{\beta}(\varphi \cdot u)$ is a function associated to a vector field, we have
$$\mathbf{P}^{\beta}(\varphi \cdot u) = a^{-1} \mathbf{P}_0 = \mathbf{P}_0.$$
In conclusion $\mathbf{P}_0$ must be $\Ort(d)$-invariant, and so $\tau= \sharp \mathbf{P}_0 \in \calT_d^k$ is an invariant $k$-tensor on $\mathbb{R}^d$. If $k$ is odd, then $\tau =0$ by Lemma~\ref{lemma:RhoSpans}, and so $P =0$. We continue with the case when $k =2p$ is even, meaning that $\beta \in DV$ is a true degree vector.

We use Lemma~\ref{lemma:RhoSpans} to write $\tau = \sum_{\rho \in \frakR_p} c_\rho \tau_\rho$ as a linear combination. The constant $\sharp \tau$ seen as a function is associated to $\sum_{\rho \in \frakR_p} c_\rho \sharp T_\rho$, and hence 
$$P f = \sum_{\rho \in \frakR_p} c_{\rho} \nabla^{\beta} f (\sharp T_{\rho}) = \sum_{\rho \in \frakR_p} c_{\rho} \tilde N(\rho, \beta) f.$$
Hence, if $\Gamma = \sum_{\rho \in \frakR_p} c_{\rho} \Gamma(\rho,\beta) \in \spn_{\mathbb{R}} \MG$, then $N\Gamma = P$, giving us surjectivity of $N$.

Finally, consider the case when $p \leq d$. Let $\Gamma \in \spn_{\mathbb{R}} \MG$ be any element such that $N\Gamma =0$. If we write $\Gamma = \sum_{\beta \in \DV} \Gamma^{\beta}$, where $\Gamma^{\beta}$ has degree vector $\beta$ then we must have $N \Gamma^{\beta} =0$ for all $\beta \in \DV$. Hence, without loss of generality we may assume that $\Gamma = \Gamma^{\beta}$ for a fixed $\beta \in \DV$. We consider $|\beta|_E =p \leq d$.
We can then write
$$\Gamma =  \sum_{\rho \in \frakR_p} c_{\rho} \Gamma(\rho, \beta), \qquad \beta = (\beta_0, \vec{\beta}).$$
By averaging over $S_{\vec{\beta}}$, we can assume that $\sum_{\rho \in \frakR_p} c_{\rho} \tau_\rho \in \Sym^{\vec{\beta}} \mathbb{R}^d$, since this averaging process does not change $\Gamma$. Furthermore,
$$N\Gamma = \sum_{\rho \in \frakR_p} c_{\rho} \tilde N(\rho, \beta) =  \nabla^{\beta} f \left( \sum_{\rho \in \frakR_p} c_{\rho} \sharp T_{\rho} \right)=0,$$
and since $\nabla^{\beta} f(x)$ can take any value in $\Sym^{\vec{\beta}} T^\ast M$, it follows that $\sum_{\rho \in \frakR_p} c_{\rho} T_{\rho} =0$. Finally, since this equation is written in a linearly independent basis, by Lemma~\ref{lemma:RhoSpans} we must have that all $c_{\rho}$ vanish and so $\Gamma=0$.\hfill$\square$

\section{Further results on equivariant operators} \label{sec:furtherResults}
\subsection{Vanishing operators from a nontrivial kernel} \label{sec:CounterEx}
We will show here the smallest dimensional case not covered by Iwahori's result in Lemma~\ref{lemma:RhoSpans}, i.e., that the set ${\{ \tau_\rho:\rho\in\mathfrak{R}_3\}}$ does not form a basis for $\calT^6_2$ of six-tensors on $\mathbb{R}^2$. In this case $\mathfrak{R}_3$ has 15 elements and in light of Lemma~\ref{lemma:RhoSpans} (b) the elements of ${\{ \tau_\rho:\rho\in\mathfrak{R}_3\}}$ span $\calT^6_2$. We can show that point (c) of the Lemma does not hold by finding a non-trivial tensor invariant to the action of $\Ort(2)$ which vanishes. Let $\tau = \sum_{\rho \in \mathfrak{R}_3} c_\rho \tau_\rho$ be a linear combination of elements $\tau_\rho$. For this to vanish, it needs to do so on each tensor $e_{i_1} \otimes \cdots \otimes e_{i_6}$, where $i_j\in\{1,2\}$, for any $j=1, \dots, 6$. We remark that if an odd number of the indices are 1 (or by symmetry 2), then all $\tau_\rho$ vanish on a tensor constructed in such a way. Notice this by defining $A=\bigl(\begin{smallmatrix}1&0\\0&-1\end{smallmatrix}\bigl)\in \Ort(2)$ and evaluating $\tau_\rho(e_{i_1},\dots,e_{i_6})=\tau_\rho(Ae_{i_1},\dots,Ae_{i_6})=-\tau_\rho(e_{i_1},\dots,e_{i_6})$ which implies that all the elements of the spanning set vanish. Thus, it only suffices to evaluate $\tau$ on tensors constructed with an even number of indices 1. Furthermore, for $A=\bigl(\begin{smallmatrix}0&1\\1&0\end{smallmatrix}\bigr)\in \Ort(2)$ for which $e_{\hat i_1} \otimes \cdots \otimes e_{\hat i_6}$ is defined such that $\hat i_j =2$ if $i_j =1$ and vice-versa, then $\tau_\rho(e_{i_1}, \dots, e_{i_6}) =\tau_\rho(e_{\hat i_1}, \dots, e_{\hat i_6})$ so that if we denote by $k$ the number of elements $e_2$, then the case with $k=0$ is the same as the case with $k=6$, and the case with $k=2$ is the same as $k=4$. Therefore, we only need to consider the cases when zero or two of the indices are equal to 2.

First, observe that for $k=0$ if $\tau =0$, then
$$\tau(e_1, \dots, e_1) = \sum_{\rho \in \mathfrak{R}_3} c_\rho =0.$$
In other words, $c = (c_\rho)$ is orthogonal to $(1, \dots, 1)^\top$ in $\mathbb{R}^{15}$. Next, for the case $k=2$ write $e_{(12)} = e_2 \otimes e_2 \otimes e_1 \otimes e_1 \otimes e_1 \otimes e_1$ as the basis vector with $e_2$ in the first and second argument, and similarly write $e_{(ij)}$ with $i \neq j$ for the vector in $(\mathbb{R}^2)^{\otimes 6}$ composed of tensor products of $e_1$ except at indices $i$ and $j$ in which it is equal to $e_2$. As $\binom{6}{2}=15$ this will result in 15 combinations.

The values of all elements $\tau_\rho$ on each $e_{(ij)}$ is given by the $15 \times 15$-matrix $M$ below.

\[
M=
\scalebox{0.8}{$
\begin{pNiceMatrix}[
    first-row, 
    first-col, 
    code-for-first-row = \bf,
    code-for-first-col = \bf
]
\rho \setminus e_{(\cdot)}  & 12 & 13 & 14 & 15 & 16 & 23 & 24 & 25 & 26 & 34 & 35 & 36 & 45 & 46 & 56\\
\{1,2\} \{3,4\} \{5,6\} &1 & 0 & 0 & 0 & 0 & 0 & 0 & 0 & 0 & 1 & 0 & 0 & 0 & 0 & 1 \\
\{1,2\}\{3,5\}\{4,6\} &1 & 0 & 0 & 0 & 0 & 0 & 0 & 0 & 0 & 0 & 1 & 0 & 0 & 1 & 0 \\
\{1,2\}\{3,6\}\{4,5\} &1 & 0 & 0 & 0 & 0 & 0 & 0 & 0 & 0 & 0 & 0 & 1 & 1 & 0 & 0 \\
\{1,3\}\{2,4\}\{5,6\} &0 & 1 & 0 & 0 & 0 & 0 & 1 & 0 & 0 & 0 & 0 & 0 & 0 & 0 & 1 \\
\{1,3\}\{2,5\}\{4,6\} &0 & 1 & 0 & 0 & 0 & 0 & 0 & 1 & 0 & 0 & 0 & 0 & 0 & 1 & 0 \\
\{1,3\}\{2,6\}\{4,5\} &0 & 1 & 0 & 0 & 0 & 0 & 0 & 0 & 1 & 0 & 0 & 0 & 1 & 0 & 0 \\
\{1,4\}\{2,3\}\{5,6\} &0 & 0 & 1 & 0 & 0 & 1 & 0 & 0 & 0 & 0 & 0 & 0 & 0 & 0 & 1 \\
\{1,4\}\{2,5\}\{3,6\} &0 & 0 & 1 & 0 & 0 & 0 & 0 & 1 & 0 & 0 & 0 & 1 & 0 & 0 & 0 \\
\{1,4\}\{2,6\}\{3,5\} &0 & 0 & 1 & 0 & 0 & 0 & 0 & 0 & 1 & 0 & 1 & 0 & 0 & 0 & 0 \\
\{1,5\}\{2,3\}\{4,6\} &0 & 0 & 0 & 1 & 0 & 1 & 0 & 0 & 0 & 0 & 0 & 0 & 0 & 1 & 0 \\
\{1,5\}\{2,4\}\{3,6\} &0 & 0 & 0 & 1 & 0 & 0 & 1 & 0 & 0 & 0 & 0 & 1 & 0 & 0 & 0 \\
\{1,5\}\{2,6\}\{3,4\} &0 & 0 & 0 & 1 & 0 & 0 & 0 & 0 & 1 & 1 & 0 & 0 & 0 & 0 & 0 \\
\{1,6\}\{2,3\}\{4,5\} &0 & 0 & 0 & 0 & 1 & 1 & 0 & 0 & 0 & 0 & 0 & 0 & 1 & 0 & 0 \\
\{1,6\}\{2,4\}\{3,5\} &0 & 0 & 0 & 0 & 1 & 0 & 1 & 0 & 0 & 0 & 1 & 0 & 0 & 0 & 0 \\
\{1,6\}\{2,5\}\{3,4\} &0 & 0 & 0 & 0 & 1 & 0 & 0 & 1 & 0 & 1 & 0 & 0 & 0 & 0 & 0 \\
\end{pNiceMatrix}
$}
\]

We can compute the null space of $M^\top$ and verify that there is a 5-dimensional subspace $K\subset\mathbb{R}^{15}$ of elements $c$ such that $M^\top c =0$, and furthermore, each row of $K$ is orthogonal to $(1, \dots, 1)$. We can write the matrix associated to $K$ in Reduced Row Echelon Form as

\[
K=
\scalebox{0.8}{$
\begin{pmatrix}
1 & 0 & -1 & 0 & 0 & 0 & -1 & 0 & 1 & 0 & 1 & -1 & 1 & -1 & 0\\
0 & 1 & -1 & 0 & 0 & 0 & 0 & 0 & 0 & -1 & 1 & 0 & 1 & -1 & 0\\
0 & 0 & 0 & 1 & 0 & -1 & -1 & 0 & 1 & 0 & 0 & 0 & 1 & -1 & 0\\
0 & 0 & 0 & 0 & 1 & -1 & 0 & 0 & 0 & -1 & 0 & 1 & 1 & 0 & -1\\
0 & 0 & 0 & 0 & 0 & 0 & 0 & 1 & -1 & 0 & -1 & 1 & 0 & 1 & -1\\

\end{pmatrix}
$}
\]
For any $c = (c_\rho) \in \mathbb{R}^{15}$, we have $\tau = \sum_{\rho \in \mathfrak{R}_3} c_\rho \tau_\rho =0$ if and only if $c \in K$. Any invariant six-tensor on $\mathbb{R}^2$ can therefore be uniquely represented by an element in $K^\perp$. For a degree vector $\beta = (\beta_0, \vec{\beta})$ such that $|\beta|_E = 3$, look at all possible values of $\vec{\beta}$. An overview is found in Table~\ref{tab:degreeVectors}.

The action of $S_{\vec{\beta}}$ on $\mathfrak{R}_3$ induces an action of $\mathbb{R}^{15}$. If $\sum_{\rho \in \frakR_3} c_\rho \Gamma(\rho, \beta)$ is nonzero, then the average of $(c_\rho)$ under the action of $S_{\vec{\beta}}$ must be non-zero as well. A quick computation reveals that all elements in $K$ vanish after $S_{\vec{\beta}}$-averaging for every degree vector except $\vec{\beta}=[0,3]$ and $\vec{\beta}=[2,2]$. For the remaining cases we will find cases where $\sum_{\rho \in \frakR_p} c_\rho \Gamma(\rho, \beta) \neq 0$, but its image under $N$ will still be zero.

\paragraph{\bf Case $\vec{\beta}=[0,3]$.} 

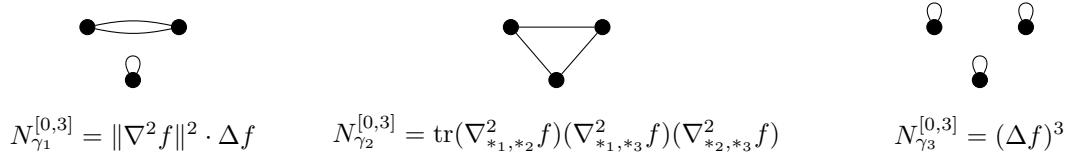
\begin{figure}[h]
\centering
\begin{tikzpicture}[scale=0.7]

\begin{scope}[xshift=0cm]
  \node[circle, draw, fill=black, inner sep=2pt] (A) at (0, 1) {};
  \node[circle, draw, fill=black, inner sep=2pt] (B) at (1.73, 1) {};
  \node[circle, draw, fill=black, inner sep=2pt] (C) at (0.86, 0) {};

  \draw (A) to[bend left=12] (B);
  \draw (A) to[bend right=12] (B);
  \draw (C) to[out=60,in=120,loop] ();
  
  \node at (0.86,-1.0) {$N^{[0,3]}_{\gamma_1}=\lVert\nabla^2 f\rVert^2\cdot\Delta f$};
\end{scope}

\begin{scope}[xshift=8cm]
  \node[circle, draw, fill=black, inner sep=2pt] (A) at (0, 1) {};
  \node[circle, draw, fill=black, inner sep=2pt] (B) at (1.73, 1) {};
  \node[circle, draw, fill=black, inner sep=2pt] (C) at (0.86, 0) {};

  \draw (A) to (B);
  \draw (B) to (C);
  \draw (C) to (A);

  \node at (0.86,-1.0) {$N^{[0,3]}_{\gamma_2}=\tr (\nabla^2_{*_1,*_2}f) (\nabla^2_{*_1,*_3}f)(\nabla^2_{*_2,*_3}f)$};
\end{scope}

\begin{scope}[xshift=16cm]
  \node[circle, draw, fill=black, inner sep=2pt] (A) at (0, 1) {};
  \node[circle, draw, fill=black, inner sep=2pt] (B) at (1.73, 1) {};
  \node[circle, draw, fill=black, inner sep=2pt] (C) at (0.86, 0) {};

  \draw (A) to[out=60,in=120,loop] ();
  \draw (B) to[out=60,in=120,loop] ();
  \draw (C) to[out=60,in=120,loop] ();

  \node at (0.86,-1.0) {$N^{[0,3]}_{\gamma_3}=(\Delta f)^3$};
\end{scope}

\end{tikzpicture}
\caption{Multigraph representatives for the degree vector $\vec \beta=[0,3]$.}
\label{fig:representatives[0,3]}
\end{figure}

There are three possible corresponding graphs representatives as in Figure~\ref{fig:representatives[0,3]}. Summing the first row of $K$ over each the $S_{\vec{\beta}}$ action gives averaged coefficients $(1,-3,2)$, yielding

\[Q^{\beta_0}f=f^{\beta_0}\cdot(-3N^{[0,3]}_{\gamma_1}+2N^{[0,3]}_{\gamma_2}+N^{[0,3]}_{\gamma_3})=0.\]
However, for $d \geq 3$, we have that $Q^{\beta_0} = N(*^{\beta_0} \cup (-3\gamma_1 + 2 \gamma_2 + \gamma_3))$ must be non-zero.

\paragraph{\bf Case $\vec{\beta}=[2,2]$.} 

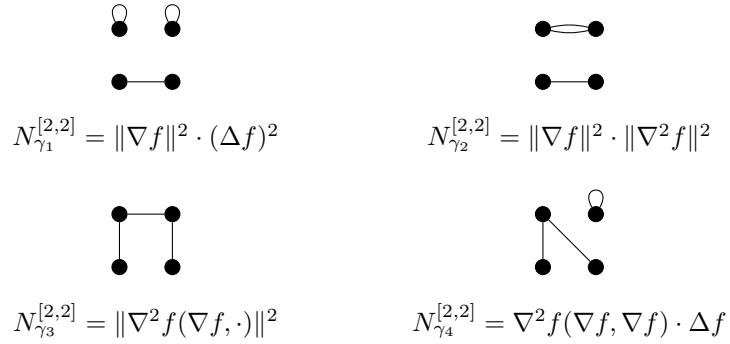
\begin{figure}[h]
\centering
\begin{tikzpicture}[scale=0.7]

\begin{scope}[xshift=0cm]
  \node[circle, draw, fill=black, inner sep=2pt] (A) at (0,0) {};
  \node[circle, draw, fill=black, inner sep=2pt] (B) at (1,0) {};
  \node[circle, draw, fill=black, inner sep=2pt] (C) at (0,1) {};
  \node[circle, draw, fill=black, inner sep=2pt] (D) at (1,1) {};

  \draw (A) to (B);
  \draw (C) to[out=60,in=120,loop] ();
  \draw (D) to[out=60,in=120,loop] ();
  
  \node at (0.5,-1.0) {$N^{[2,2]}_{\gamma_1}=\lVert\nabla f\rVert^2 \cdot (\Delta f)^2$};
\end{scope}

\begin{scope}[xshift=8cm]
  \node[circle, draw, fill=black, inner sep=2pt] (A) at (0,0) {};
  \node[circle, draw, fill=black, inner sep=2pt] (B) at (1,0) {};
  \node[circle, draw, fill=black, inner sep=2pt] (C) at (0,1) {};
  \node[circle, draw, fill=black, inner sep=2pt] (D) at (1,1) {};

  \draw (C) to[bend left=12] (D);
  \draw (C) to[bend right=12] (D);
  \draw (A) to (B);

  \node at (0.5,-1.0) {$N^{[2,2]}_{\gamma_2}=\lVert\nabla f\rVert^2\cdot\lVert\nabla^2 f\rVert^2$};
\end{scope}

\begin{scope}[xshift=0cm, yshift=-35mm]
  \node[circle, draw, fill=black, inner sep=2pt] (A) at (0,0) {};
  \node[circle, draw, fill=black, inner sep=2pt] (B) at (1,0) {};
  \node[circle, draw, fill=black, inner sep=2pt] (C) at (0,1) {};
  \node[circle, draw, fill=black, inner sep=2pt] (D) at (1,1) {};

  \draw (B) to (D);
  \draw (A) to (C);
  \draw (C) to (D);

  \node at (0.5,-1.0) {$N^{[2,2]}_{\gamma_3}=\lVert\nabla^2 f(\nabla f, \cdot)\rVert^2$};
\end{scope}

\begin{scope}[xshift=8cm, yshift=-35mm]
  \node[circle, draw, fill=black, inner sep=2pt] (A) at (0,0) {};
  \node[circle, draw, fill=black, inner sep=2pt] (B) at (1,0) {};
  \node[circle, draw, fill=black, inner sep=2pt] (C) at (0,1) {};
  \node[circle, draw, fill=black, inner sep=2pt] (D) at (1,1) {};

  \draw (A) to (C);
  \draw (B) to (C);
  \draw (D) to[out=60,in=120,loop] ();

  \node at (0.5,-1.0) {$N^{[2,2]}_{\gamma_4}=\nabla^2 f(\nabla f, \nabla f)\cdot \Delta f$};
\end{scope}

\end{tikzpicture}
\caption{Multigraph representatives associated to the orbits for the degree vector $\vec \beta =[2,2]$.}
\label{fig:representatives[2,2]}

\end{figure}

There are four possible graphs as in Figure~\ref{fig:representatives[2,2]}. Summing the first row of $K$ over $S_{\vec{\beta}}$ gives averaged coefficients $(1,-1,2,-2)$, yielding
\[P^{\beta_0}f=f^{\beta_0}\cdot(N^{[2,2]}_{\gamma_1} - N^{[2,2]}_{\gamma_2} + 2N^{[2,2]}_{\gamma_3} - 2N^{[2,2]}_{\gamma_4}) = 0,\]
but with $P^{\beta_0} = N(*^{\beta_0} \cup (\gamma_1 - \gamma_2 + 2 \gamma_3 - 2 \gamma_4)$ is nonvanishing for $d \geq 3$.

In both cases, the remaining kernel rows give coefficient vectors proportional to these when averaging over $S_{\vec{\beta}}$, so each degree vector produces exactly one independent vanishing operator. All of these observations in summary give the proof of Theorem~\ref{th:PQvanishes}.

The two vanishing operators $P^{\beta_0}$ and $Q^{\beta_0}$ identified above are not merely computational curiosities: they represent the complete obstruction to injectivity of the map $N$ in the first dimension where Iwahori's basis theorem fails. Hence, if we want to write a basis of equivariant nonlinear operators up to total order six, we need these relations. We can continue this procedure for any dimension as follows:
\begin{enumerate}[$\bullet$]
\item When $p > d$, look at the kernel $K$ of the mapping from $\spn_{\mathbb{R}} \frakR_p \to \calT_d^{2p}$ given by $\sum_{\rho \in \frakR_p} c_\rho \rho \mapsto \sum_{\rho\in \frakR_p} c_\rho \tau_\rho$.
\item For every $\beta =(\beta_0, \vec{\beta}) \in \DV$ with $|\beta|_E = p$, consider $K^{S_{\vec{\beta}}}$. For any nonzero $w =\sum_{\rho \in \frakR_p} c_\rho \rho \neq 0$ in $K^{S_{\vec{\beta}}}$, we get a linear dependence relation
$$\sum_{\rho \in \frakR_p} c_\rho \tilde N(\rho, \beta) =0,$$
will be a dimension-dependent relation between nonlinear operators.
\end{enumerate}

\begin{remark}
The vanishing of $Q^{\beta_0}$ on two-dimensional manifolds has a classical explanation via the Cayley--Hamilton theorem. In dimension $d=2$, the Hessian $A = \nabla^2 f$ is a $2\times 2$ symmetric matrix with characteristic polynomial $\chi_A(\lambda) = \lambda^2 - (\tr A)\lambda + \det A$, so Cayley--Hamilton gives $A^2 = (\tr A)A - (\det A)I$. Taking the trace of the identity $A^3 =(\tr A)A^2 - (\det A)A$ and using Newton's identity $\det A = \tfrac{1}{2}\bigl[(\tr A)^2 - \tr(A^2)\bigr]$ yields
\[
    \tr(A^3) \;=\; \tfrac{3}{2}(\tr A)\tr(A^2) - \tfrac{1}{2}(\tr A)^3.
\]
Substituting $\tr A = \Delta f$, $\tr(A^2) = \|\nabla^2 f\|^2$ and $\tr(A^3) = \tr\bigl((\nabla^2 f)^3\bigr)$, this rearranges to
\[
    (\Delta f)^3 - 3\|\nabla^2 f\|^2\,\Delta f + 2\tr\bigl((\nabla^2 f)^3\bigr) = 0,
\]
which is equivalent to $Q^{\beta_0}f = 0$.
\end{remark}

\subsection{Remarks on less symmetry} \label{sec:Lesssymmetry}
We consider the case when we are considering tensors that are only invariant with respect to the connected isometry group. Following  \cite{iwahori1958some,weyl1946classical}, we are now considering the space of $p$-tensors $\hat \calT_d^p$ on $\mathbb{R}^d$ that are invariant under $\SO(d)$. Clearly, $\calT_d^p \subset \hat \calT_d^{p}$ in the notation of Section~\ref{sec:Tensors}, but we will have the following tensors in addition. Let $\hat \frakR_{p,d}$ be the collection of unordered pairs $\hat \rho = \{ \{ i_1, i_2\}, \dots, \{i_{2p-1}, i_{2p}\}\}$ where each integer appears at most only once, but now the integers may be taken from $1, \dots,2p+d$, meaning that there is a set $\hat \rho^\perp = \{ j_1 <\dots< j_d\}$ of elements not appearing in any pair.
We can then define $\hat \tau_{\hat \rho} \in \hat \calT^{2p+d}$, $\hat \rho \in \hat\frakR_{q,d}$ by
$$\hat \tau_{\hat \rho}(x_1, \dots, x_{2p+d}) = \det(x_{j_1}, \dots, x_{j_d}) \cdot \prod_{\{ i,i_2\} \in \rho} \langle x_{i_1}, x_{i_2} \rangle, \qquad \hat \rho^\perp = \{ j_1 <\dots< j_d\} .$$
If $M$ is oriented with volume form $\omega$, we can define the corresponding map on manifolds,
$$\hat T_{\hat \rho}(x_1, \dots, x_{2p+d}) = \omega(x_{j_1}, \dots, x_{j_d}) \cdot \prod_{\{ i,i_2\} \in \rho} \langle x_{i_1}, x_{i_2} \rangle_g, \qquad \hat \rho^\perp = \{ j_1 <\dots< j_d\} ,$$
Let $\beta \in\hDV$ be pseudo-degree vector with $|\beta|_{E} = \frac{1}{2}(2p+d)$ and with $\hat \rho \in \hat \frakR_{p,d}$. We define
$$\hat N(\hat \rho, \beta)f = \nabla^{\beta}f(\sharp T_{\hat \rho}),$$
which is a differential operator of total order $2p+d$.

In order to construct the analogue of the multigraph representation, from the pair $(\hat \rho, \beta)$, if $|\beta|_V = n$, then we consider the map $m_{\beta}: \{ 1, \dots, 2p+d\} \to \{ v_1, \dots, v_n\}$ as earlier, and define edges $E=\{ \{ m(i), m(i_2) \, : \, \{ i, i_2 \} \in \hat \rho \}$, giving us a graph $\gamma = (V, E)$. Introduce then also a map $\phi: \{ 1, \dots, d\}\to V$ by $\phi(r) = m_{\beta}(j_r)$.
Notice that if $\phi(r_1) = \phi(r_2)$ for $r_1 \neq r_2$, we get $\hat N(\hat \rho, \beta) =0$, as part of the inner product $\langle \nabla^{\beta}f, T_{\hat \rho} \rangle_g$ will involve a part where a symmetric tensor is multiplied by a skew-symmetric tensor. Hence, we can restrict ourselves to the case where $\phi$ is injective. We can then define the following classifying space, using two equivalence relations.
\begin{enumerate}[$\bullet$]
\item Consider pair $(\phi, \gamma)$ where $\gamma$ is a multigraph and $\phi$ is an injective map into the vertex set~$V[\gamma]$. If $\gamma$ and $\tilde \gamma$ are isomorphic by the map $\Phi: V[\gamma] \to V[\tilde \gamma]$, we define an equivalence relation $\sim$ by $(\phi, \gamma) \sim (\phi \circ \Phi, \tilde \gamma)$. Define $\widehat{\mathrm{ExtMG}} = \spn_{\mathbb{R}} \{[\phi,\gamma]_\sim\}$, where the set is spanned over all equivalent classes.
\item Define $[\phi, \gamma] = [\phi,\gamma]_{\simeq}$ as the equivalence class of element $[\phi, \gamma]_\sim$ in $\mathrm{ExtMG} = \widehat{\mathrm{ExtMG}}/\simeq$, where we define $[\phi \circ \sigma , \gamma]_{\sim} \simeq (-1)^{\mathrm{sign}(\sigma)} [\phi, \gamma]_{\sim}$ for any $\sigma \in S_d$.
\end{enumerate}
Let
$$\hat N([\phi,\gamma])f = \tr_E \tr_\phi \nabla^{\phi,V} f, \qquad f\in C^\infty(M),$$
where
$$\nabla^{\phi,V} f = \nabla^{\phi,v_1} f \otimes \cdots \otimes \nabla^{\phi,v_n}f, \qquad \text{ with  } \qquad \nabla^{\phi,v_i} f = \nabla^{\deg(v_i)+|\phi^{-1}(v_i)|,\Sym} f,$$
$\tr_E$ takes traces between $\nabla^{\phi,v_{i_1}}$ and $\nabla^{\phi,v_{i_2}}$ for every edge $e= \{ v_{i_1}, v_{i_2} \}$, and $\tr_{\phi}$ takes the trace of the $j$-th position of $\omega$ and of $\nabla^{\phi, \phi(j)}$ for any $j =1, \dots, d$.

By the same proof as for Theorem~\ref{th:main}, the map $\hat N$ will be surjective whenever the isometry group acts transitively on oriented orthonormal frame bundle $\SO(M) \to M$ consisting of frames with positive orientation. By \cite[Notes 10]{kobayashi1996foundations}, these are, up to scaling, the same Riemannian model spaces $\mathbb{R}^n$, $\mathbb{S}^n$ and $\mathbb{H}^n$, but the result now also holds for scalings of the real projective space $\mathbb{RP}^n$.

\begin{example} \label{ex:SO}
\begin{enumerate}[\rm (a)]
\item In dimension $d=2$, we have that $\omega(\xi_1, \xi_2) = \langle J\xi_1 , \xi_2 \rangle$ where $J$ is rotation by $\frac{\pi}{2}$ in a positively oriented frame. We then can consider the operator $\hat N([\phi, \gamma])f = \nabla^2 f(J\nabla f, \nabla f) = \tr \nabla^2 f(\ast _1,\ast _2) \omega(\ast _3,\ast _1) \nabla f(\ast _3) \nabla f(\ast _2)$, where the traces are taken in positions of equal numbers. Then $\gamma$ can be given as the multigraph with $V_\gamma = \{v_1, v_2, v_3\}$, $E_\gamma = \{\{ v_1, v_2\}, \{ v_2, v_3\}\}$ with $\phi(1) = v_1$ and $\phi(2) = v_2$.
\item In dimension $d=3$, we can define a cross product by $\omega(\xi_1, \xi_2, \xi_3) \} = \langle \xi_1 \times \xi_2, \xi_3 \rangle$. We can then consider an operator defined using $\hat N([\phi, \gamma]) f= \langle \nabla f \times \sharp \nabla^2 f(\nabla f, \cdot), \nabla \Delta f\rangle_g$, where $V_\gamma = \{ v_1, v_2, v_3, v_4\}$, $E_\gamma = \{\{ v_2,v_3\}, \{ v_4, v_4\}\}$ and $\phi(1)=v_1$, $\phi(2) = v_2$, $\phi(3) = v_4$.
\end{enumerate}
\end{example}

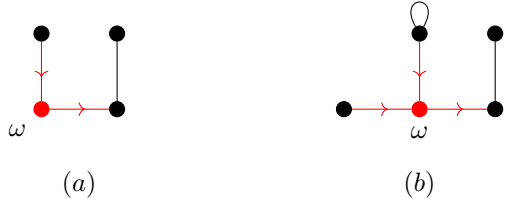
\begin{figure}[h]
\centering
\begin{tikzpicture}[scale=1.0]

\begin{scope}[xshift=0cm]
  \node[circle, draw, fill=black, inner sep=2pt] (A) at (0, 1) {};
  \node[circle, draw=red, fill=red, inner sep=2pt, label=below left:$\omega$] (B) at (0, 0) {};
  \node[circle, draw, fill=black, inner sep=2pt] (C) at (1, 0) {};
  \node[circle, draw, fill=black, inner sep=2pt] (D) at (1, 1) {};

  \draw[
    red,
    postaction={
        decorate,
        decoration={
            markings,
            mark=at position 0.6 with {\arrow{>}}
        }
    }
  ] (A) to (B);
  \draw[
    red,
    postaction={
        decorate,
        decoration={
            markings,
            mark=at position 0.6 with {\arrow{>}}
        }
    }
  ] (B) to (C);
  \draw (C) to (D);
  
  \node at (0.5,-1.0) {$(a)$};
\end{scope}

\begin{scope}[xshift=4cm]
  \node[circle, draw, fill=black, inner sep=2pt] (A) at (0, 0) {};
  \node[circle, draw=red, fill=red, inner sep=2pt, label=below:$\omega$] (B) at (1, 0) {};
  \node[circle, draw, fill=black, inner sep=2pt] (C) at (2, 0) {};
  \node[circle, draw, fill=black, inner sep=2pt] (D) at (1, 1) {};
  \node[circle, draw, fill=black, inner sep=2pt] (E) at (2, 1) {};

  \draw[
    red,
    postaction={
        decorate,
        decoration={
            markings,
            mark=at position 0.6 with {\arrow{>}}
        }
    }
  ] (A) to (B);
  \draw[
    red,
    postaction={
        decorate,
        decoration={
            markings,
            mark=at position 0.6 with {\arrow{>}}
        }
    }
  ] (B) to (C);
  \draw[
    red,
    postaction={
        decorate,
        decoration={
            markings,
            mark=at position 0.6 with {\arrow{>}}
        }
    }
  ] (D) to (B);
  \draw (E) to (C);
  \draw (D) to[out=60,in=120,loop] ();

  \node at (1,-1.0) {$(b)$};
\end{scope}

\end{tikzpicture}
\caption{A graphical way to represent the operators of Example~\ref{ex:SO}. We have an edge for every trace, with the red edges going into the vertex $\omega$ denoting the traces with the volume form. These graphs are directed so that arrows pointing into $\omega$ denotes taking a trace at an odd argument, while an arrow away from~$\omega$ denotes a trace in an even input. There must be $0$ or $1$ more arrows towards~$\omega$ when respectively~$d$ is even or odd. Interchanging direction of two arrows, one going to~$\omega$, one away, gives the same operator but with opposite signs. The polynomial degree equals the number of black vertices, and the total order equals the sum of all degrees of the same vertices.}
\label{fig:directed_graphs}
\end{figure}

\begin{remark}
One might be tempted to include a tensor with an even number of determinants to obtain something that is invariant under the full isometry group, but, as pointed out in \cite{weyl1946classical}, such tensors do not give us anything new, since
$$\det(x_1, \dots, x_d) \cdot \det(y_1, \dots, y_d) = \det(\langle x_i, y_j \rangle_{ij}) \in \spn \{ \tau_\rho \, : \, \rho \in \frakR_{d} \}.$$
\end{remark}

\subsection{Sub-Riemannian model spaces}
A sub-Riemannian manifold $(M,E,g)$ is a triple, where~$M$ is a connected manifold, $E \subseteq TM$ is a subbundle of $TM$ and $g$ is a metric tensor defined only on~$E$. In this section, we will assume that $E$ is \emph{bracket-generating, equiregular}, meaning that there exists a flag of subbundles
$$E =E^1 \subsetneq E^2 \subsetneq \cdots \subsetneq E^s = TM,$$
with $E^{i+1} = \spn\{ E^i, [E,E^i] \}$. We write $d_1$ for the rank of $E$. Note that the case $s=1$ equals a Riemannian manifold. By the bracket-generating condition, we can write any linear differential operators using only vector fields with values in $E$. We say that a linear differential operator $L$ has \emph{nonholonomic order} $k$ if $k$ is the minimal number of derivatives in the direction of $E$ needed for $L$, see \cite[Section~4.2]{bellaiche1996tangent} for details.

A connection $\nabla$ is said to be \emph{compatible} with the sub-Riemannian structure $(E,g)$ if parallel transport with respect to the connection takes orthonormal frames of $E$ to orthonormal frames. This is equivalent to requiring that for any $X, Y \in \Gamma(E)$ and $Z \in \Gamma(TM)$, then $\nabla_Z X$ will take values in $E$, and $Z\langle X, Y \rangle_g = \langle \nabla_Z X, Y \rangle + \langle Y, \nabla_Z Y \rangle_g$. We can also formulate this property using the sub-Riemannian \emph{cometric}. Let $\sharp: TM \to E$ be the map $\sharp \alpha(\xi) = \langle \xi , \sharp \alpha\rangle_g$ for $\alpha \in T^\ast M$, $\xi \in TM$, which will have a kernel when $s >1$. We can then define the (degenerate when $s >1$) \emph{sub-Riemannian cometric} $g^\ast $ by
$$\langle \alpha_1, \alpha_2 \rangle_{g^\ast } = \langle\sharp \alpha_1, \sharp \alpha_2 \rangle_g.$$
The connection $\nabla$ is compatible with the sub-Riemannian structure $(E,g)$ if and only if $\nabla g^\ast  =0$. For more on connections and sub-Riemannian structures, see, e.g., \cite{grong2026canonical}.

With such a connection, we can write any linear differential operator as linear combinations of operator
$$\nabla^k_{X} f= \nabla^kf(X), \qquad X\in \Gamma(E^{\otimes k}),$$
where the above has nonholonomic order $k$.
Note that we are not using the symmetric covariant derivatives here, since we are no longer ensured that alternating parts will have lower sub-Riemannian order. For example, if $X,Y \in \Gamma(E)$ and $[X,Y]$ does not take values in $E$, then $-\nabla_{Tor(X,Y)} f = \nabla^2_{X,Y} - \nabla_{Y,X}^2f$ will be a differential operator of nonholonomic order $2$, where $Tor$ is the torsion of $\nabla$. We note that $[X,Y]f$ equals $- \nabla_{Tor(X,Y)}f$ up to differential operators of nonholonomic order 1 by the compatibility of the connection.

We can define the principal bundle $\Ort(d_1) \to \Ort(E)\to M$ as the bundle of orthonormal frames $u: \mathbb{R}^{d_1} \to E_{x}$ of $E$, with $\Ort(d_1)$ acting on each fiber $\Ort(E)_x$, $x \in M$ on the right by composition. \emph{A sub-Riemannian isometry} $\varphi: M \to M$ is a diffeomorphism satisfying $\varphi_\ast (E) \subseteq E$ and mapping $E_x$ isometrically onto $E_{\varphi(x)}$ for any $x \in M$. Write $G_{sr} = G(M,E,g)$ for the group of sub-Riemannian isometries. We say that that $(M,E,g)$ is \emph{a sub-Riemannian model space in the sense of isometries} if $G_{sr}$ acts transitively on $\Ort(E)$. For more on such spaces, see \cite{grong2021model,berge2021g}. In particular, if $\nabla$ is a compatible connection on such a sub-Riemannian model space, then requiring that $\varphi_\ast  \nabla_{Y_1} Y_2 = \nabla_{\varphi_\ast  Y_1} \varphi_\ast  Y_2$ for any $Y_1, Y_2 \in \Gamma(E)$, $\varphi \in G_{sr}$ uniquely determines $\nabla_{Y_1} Y_2$ for $Y_1, Y_2$. See \cite[Prop~3.3.]{grong2021model} where this result is written in the language of partial connections. We will say that such a connection is \emph{equivariant} under $G_{sr}$. It follows that on a sub-Riemannian model space, the operator $\nabla^k_X$, $X \in \Gamma(E^{\otimes k})$ will be the same for every isometry-equivariant connection~$\nabla$.

For a degree vector $\beta = (\beta_0,\vec{\beta})$ of order $k$, define now
$$\nabla^{\beta} f = f^{\beta_0} \cdot  \underbrace{\nabla f \otimes \cdots \otimes \nabla f}_{\text{$\vec{\beta}(1)$ times}} \otimes \cdots \otimes \underbrace{\nabla^k f \otimes \cdots \otimes \nabla^k f}_{\text{$\vec{\beta}(k)$ times}} ,$$
that is, defined now with just the usual covariant derivative rather than the symmetrized ones. Also define $T_\rho \in \Gamma(TM^{2p})$, $\rho \in \frakR_p$ now as a tensor on covectors by
$$T_\rho(\alpha_1, \dots, \alpha) = \prod_{\{i,j\} \in \rho} \langle \alpha_i, \alpha_j\rangle_{g^\ast }.$$
Using the same argument as in the first part of the proof of Theorem~\ref{th:main} in Section~\ref{sec:ProofMain}, we have that $\NPDO(M)^{G_{sr}}$ is spanned by operators on the form
$$P =\tilde N(\rho, \beta): f \mapsto (\nabla^{\beta}f)(T_\rho),$$
with $\beta$ being a degree vector with $|\beta|_E = p$ and $\rho \in \frakR_{2p}$. We note that $|\beta|_V =n$ is still the polynomial degree of $P$, and if we define order and total order according nonholonomic order, then the order of $\beta$ and $2|\beta|_E = 2p$ are respectively the nonholonomic order and nonholonomic total order of $P$.

It follows $\NPDO(M)^{G_{sr}} = \spn \{\tilde N(\rho,\beta)\}$, but determining linear dependence relations between these operators is more difficult, since we no longer have the correspondence with multigraphs. Even determining the linear equivariant operators $f\mapsto \nabla^{2p}f(T_\rho)$ is a non-trivial task, which we will leave for future research, only including the simplest example below.

\begin{example}
If we look at the simplest non-Riemannian example, consider the three dimensional Lie algebra $\frakg_\kappa$, $\kappa \in \mathbb{R}$, with basis $X,Y, Z$ and relations
$$[X,Y] = Z, \qquad [Y,Z] = \kappa X, \qquad [Z,X] = \kappa Y.$$
We remark that the cases $\kappa >0$, $\kappa <0$ and $\kappa =0$ are isomorphic to respectively $\so(3) \cong \mathfrak{su}(2)$, $\mathfrak{sl}(2)$ and the Heisenberg algebra of dimension 3. Let $M_\kappa$ be the corresponding simply connected Lie group. We will also use $X,Y, Z$ for the corresponding left-invariant vector fields on $M_\kappa$. Define a sub-Riemannian structure $(E,g)$ on $M_\kappa$ by defining $X,Y$ to be an orthonormal basis. These spaces are then all the sub-Riemannian model spaces with a manifold of dimension $3$ and with $E$ of rank 2 up to isometry, see \cite{grong2021model} for details.

Left translations will be isometries of $(M_\kappa, E,g)$, but also for any $A = (A_{ij} )\in \Ort(2)$, there is a corresponding sub-Riemannian isometry $\varphi_A$ satisfying $\varphi_A(1) =1$ and
\begin{equation} \label{ASR}(\varphi_A)_\ast  X = A_{11} X + A_{21} Y, \qquad (\varphi_A)_\ast  Y = A_{12} X + A_{22} Y, \qquad (\varphi_A)_\ast  Z = \det(A) Z.\end{equation}
These two types of isometries together generate the group of isometries $G_{sr}$.
If we define $\nabla$ as the connection $\nabla X = \nabla Y = \nabla Z = 0$, then this connection will be equivariant under $G_{sr}$.

The sub-Laplacian
$$\Delta_{sr} f = \tr_{H} \nabla^2_{\ast ,\ast } f = (X^2 + Y^2)f = \nabla^2 f(T_{\{1,2\}}),$$
is equivariant with respect to $G_{sr}$, but also
$$Z^2 f = \nabla^4 f (T_{\{1,3\}, \{2,4\}} - T_{\{1,4\}\{2,3\}}).$$
These two operators commute since $[Z,\Delta_{sr}] =0$. Any left-invariant linear differential operators can be written as a polynomial in $X, Y, Z$, so we can see that elements invariant under the action in \eqref{ASR} can be written as $p(\Delta_{sr}, Z^2)$, where $p(x,y)$ is a polynomial in two variables, paralleling Helgason's result \cite{helgason1959differential} in the Riemannian case.

For nonlinear operators, we note that equivariant operators such as
$$f \mapsto Zf(X f Y\Delta_{sr} f - X \Delta_{sr} f Yf) = \nabla f \otimes \nabla^2 f \otimes \nabla^3 f(T_{\{1,2\}\{3,4\}\{5,6\}} -T_{\{1,3\}\{2,4\}\{5,6\}})$$
have no analogue in the Riemannian case.
\end{example}

\bibliographystyle{abbrv}
\bibliography{Bibliography}

\newpage
\appendix
\section{Tables}
\begin{table}[ht]
\centering
\begin{tabular}{%
  |m{15mm}
  m{30mm}
  c
  m{50mm}|
}

\hline
\# edges & Graph class  & Representative & Associated operator \\ 
\hline
$0$ &
empty graph &
$\emptyset$&
$1$ \\[0.5em]

$0$ &
isolated vertex &
\begin{tikzpicture}[scale=1.1]
  \node[circle, draw, fill=black, inner sep=2pt] (A) at (0,0) {};
\end{tikzpicture}&
$\nabla^0f = f$ \\[0.5em]

\hline

$1$ &
single edge &
\begin{tikzpicture}[scale=1.1]
  \node[circle, draw, fill=black, inner sep=2pt] (A) at (0,0) {};
  \node[circle, draw, fill=black, inner sep=2pt] (B) at (0.6,0) {};
  \draw (A) -- (B);
\end{tikzpicture}&
$\lVert \nabla f\rVert^2$ \\[0.5em]

$1$ &
single loop &
\begin{tikzpicture}[scale=1.1]
  \node[circle, draw, fill=black, inner sep=2pt] (A) at (0,0) {};
  \draw (A) to[out=60,in=120,loop] ();
\end{tikzpicture}&
$\tr(\nabla^2f)=\Delta f$ \\

\hline
$2$ &
path of length two &
\begin{tikzpicture}[scale=1.1]
  \node[circle, draw, fill=black, inner sep=2pt] (A) at (0,0) {};
  \node[circle, draw, fill=black, inner sep=2pt] (B) at (0.6,0) {};
  \node[circle, draw, fill=black, inner sep=2pt] (C) at (1.2,0) {};
  \draw (A) to (B);
  \draw (B) to (C);
\end{tikzpicture}&
$\nabla^2 f(\nabla f, \nabla f)$ \\

$2$ &
double edge &
\begin{tikzpicture}[scale=1.1]
  \node[circle, draw, fill=black, inner sep=2pt] (A) at (0,0) {};
  \node[circle, draw, fill=black, inner sep=2pt] (B) at (1,0) {};
  \draw (A) to[bend left=12] (B);
  \draw (A) to[bend right=12] (B);
\end{tikzpicture}&
$\| \nabla^2 f\|^2$ \\

$2$ &
loop and edge &
\begin{tikzpicture}[scale=1.1]
  \node[circle, draw, fill=black, inner sep=2pt] (A) at (0,0) {};
  \node[circle, draw, fill=black, inner sep=2pt] (B) at (0.6,0) {};
  \draw (A) -- (B);
  \draw (A) to[out=60,in=120,loop] ();
\end{tikzpicture} &
$\langle \nabla \Delta f, \nabla f \rangle$ \\

$2$ &
double loop &
\begin{tikzpicture}[scale=1.1]
  \node[circle, draw, fill=black, inner sep=2pt] (A) at (0,0) {};

  \draw (A) to[out=20,in=80,loop] ();
  \draw (A) to[out=100,in=160,loop] ();
\end{tikzpicture}&
$\Delta^2 f$ \\
\hline

$3$ &
path of length three &
\begin{tikzpicture}[scale=1.1]
  \node[circle, draw, fill=black, inner sep=2pt] (A) at (0,0) {};
  \node[circle, draw, fill=black, inner sep=2pt] (B) at (0.6,0) {};
  \node[circle, draw, fill=black, inner sep=2pt] (C) at (1.2,0) {};
  \node[circle, draw, fill=black, inner sep=2pt] (D) at (1.8,0) {};
  \draw (A) to (B);
  \draw (B) to (C);
  \draw (C) to (D);
\end{tikzpicture}&
$\|\nabla^2f(\nabla f, \cdot)\|^2$ \\

$3$ &
star &
\begin{tikzpicture}[scale=1.1]
  \node[circle, draw, fill=black, inner sep=2pt] (A) at (0,0) {};
  \node[circle, draw, fill=black, inner sep=2pt] (B) at (0.5,0) {};
  \node[circle, draw, fill=black, inner sep=2pt] (C) at (1,0) {};
  \node[circle, draw, fill=black, inner sep=2pt] (D) at (0.5,0.5) {};
  \draw (A) to (B);
  \draw (B) to (C);
  \draw (B) to (D);
\end{tikzpicture}&
$\nabla^3f (\nabla f,\nabla f,\nabla f) $\\

$3$ &
triangle &
\begin{tikzpicture}[scale=1.1]
  \node[circle, draw, fill=black, inner sep=2pt] (A) at (0,0) {};
  \node[circle, draw, fill=black, inner sep=2pt] (B) at (0.6,0) {};
  \node[circle, draw, fill=black, inner sep=2pt] (C) at (0.3,0.5196) {};
  \draw (A) to (B);
  \draw (B) to (C);
  \draw (C) to (A);
\end{tikzpicture}&
$\tr (\nabla^2_{*_1,*_2}f) (\nabla^2_{*_1,*_3}f)(\nabla^2_{*_2,*_3}f)$\\

$3$ &
path of length two with a double edge &
\begin{tikzpicture}[scale=1.1]
  \node[circle, draw, fill=black, inner sep=2pt] (A) at (0,0) {};
  \node[circle, draw, fill=black, inner sep=2pt] (B) at (0.6,0) {};
  \node[circle, draw, fill=black, inner sep=2pt] (C) at (1.2,0) {};
  \draw (A) to[bend left=12] (B);
  \draw (A) to[bend right=12] (B);
  \draw (B) to (C);
\end{tikzpicture}&
$\langle \nabla^{3, \Sym} f, \nabla f \otimes \nabla^2 f \rangle_{g^*}$ \\

$3$ &
path of length two and one loop on the side &
\begin{tikzpicture}[scale=1.1]
  \node[circle, draw, fill=black, inner sep=2pt] (A) at (0,0) {};
  \node[circle, draw, fill=black, inner sep=2pt] (B) at (0.6,0) {};
  \node[circle, draw, fill=black, inner sep=2pt] (C) at (1.2,0) {};
  \draw (A) -- (B);
  \draw (B) -- (C);
  \draw (A) to[out=60,in=120,loop] ();
\end{tikzpicture}&
$\nabla^2 f(\nabla \Delta f, \nabla f)$ \\

$3$ &
path of length two and loop in the middle &
\begin{tikzpicture}[scale=1.1]
  \node[circle, draw, fill=black, inner sep=2pt] (A) at (0,0) {};
  \node[circle, draw, fill=black, inner sep=2pt] (B) at (0.6,0) {};
  \node[circle, draw, fill=black, inner sep=2pt] (C) at (1.2,0) {};
  \draw (A) -- (B);
  \draw (B) -- (C);
  \draw (B) to[out=60,in=120,loop] ();
\end{tikzpicture}&
$\nabla^2\Delta f(\nabla f, \nabla f) $ \\

$3$ &
triple-edge &
\begin{tikzpicture}[scale=1.1]
  \node[circle, draw, fill=black, inner sep=2pt] (A) at (0,0) {};
  \node[circle, draw, fill=black, inner sep=2pt] (B) at (1.0,0) {};
  \draw (A) to (B);
  \draw (A) to[bend left=20] (B);
  \draw (A) to[bend right=20] (B);
\end{tikzpicture}&
$\lVert \nabla^{3,\Sym}f\rVert^2$ \\

$3$ &
double-edge with a loop &
\begin{tikzpicture}[scale=1.1]
  \node[circle, draw, fill=black, inner sep=2pt] (A) at (0,0) {};
  \node[circle, draw, fill=black, inner sep=2pt] (B) at (1,0) {};
  \draw (A) to[bend left=12] (B);
  \draw (A) to[bend right=12] (B);
  \draw (A) to[out=60,in=120,loop] ();
\end{tikzpicture}&
$\langle\nabla^2\Delta f, \nabla^2 f\rangle$\\

$3$ &
one edge and one loop per side &
\begin{tikzpicture}[scale=1.1]
  \node[circle, draw, fill=black, inner sep=2pt] (A) at (0,0) {};
  \node[circle, draw, fill=black, inner sep=2pt] (B) at (0.6,0) {};
  \draw (A) -- (B);
  \draw (A) to[out=60,in=120,loop] ();
  \draw (B) to[out=60,in=120,loop] ();
\end{tikzpicture} &
$\lVert \nabla \Delta f \rVert^2$ \\

$3$ &
edge and two loops on the same side &
\begin{tikzpicture}[scale=1.1]
  \node[circle, draw, fill=black, inner sep=2pt] (A) at (0,0) {};
  \node[circle, draw, fill=black, inner sep=2pt] (B) at (0.6,0) {};

  \draw (A) to (B);
  \draw (A) to[out=20,in=80,loop] ();
  \draw (A) to[out=100,in=160,loop] ();
\end{tikzpicture}&
$\langle\nabla\Delta^2f, \nabla f\rangle$\\

$3$ &
triple loop &
\begin{tikzpicture}[scale=1.1]
  \node[circle, draw, fill=black, inner sep=2pt] (A) at (0,0) {};

  \draw (A) to[out=0,in=60,loop] ();
  \draw (A) to[out=60,in=120,loop] ();
  \draw (A) to[out=120,in=180,loop] ();
\end{tikzpicture}&
$\Delta^3f$\\

\hline
\end{tabular}
\caption{
All isomorphism classes of multigraphs of order $|\gamma| \le 3$ and their associated
invariant differential operators on the flat space. The same classes also hold for for non-flat spaces up to lower-order terms.
}
\end{table} \label{sec:GraphDiffO}
\newpage
\begin{table}[H]
\centering
\begin{tabular}{%
  |m{15mm}
  m{35mm}
  m{40mm}
  m{20mm}|
}

\hline
\# nodes & Graph class  & Representative & Symmetry group \\ 
\hline
$1$ &
node &
\begin{tikzpicture}[scale=1.1]
  \node[circle, draw, fill=black, inner sep=2pt, label={$2n_1$}] (A) at (0,0) {};
\end{tikzpicture}&
$Id$ \\[0.5em]

\hline

$2$ &
path &
\begin{tikzpicture}[scale=1.1]
  \node[circle, draw, fill=black, inner sep=2pt, label={$2n_1$}] (A) at (0,0) {};
  \node[circle, draw, fill=black, inner sep=2pt, label={$2n_2$}] (B) at (1.0,0) {};

  \draw (A) -- (B) node [midway, below] {$e_1$};
\end{tikzpicture}&
$\mathbb Z_2$\\[0.5em]

\hline

$3$ &
path&
\begin{tikzpicture}[scale=1.1]
  \node[circle, draw, fill=black, inner sep=2pt, label={$2n_1$}] (A) at (0,0) {};
  \node[circle, draw, fill=black, inner sep=2pt, label={$2n_2$}] (B) at (0.8,0) {};
  \node[circle, draw, fill=black, inner sep=2pt, label={$2n_3$}] (C) at (1.6,0) {};

  \draw (A) -- (B) node [midway, below] {$e_1$};
  \draw (B) -- (C) node [midway, below] {$e_2$};
\end{tikzpicture}  &
$\mathbb Z_2$\\

$3$ &
triangle &
\begin{tikzpicture}[scale=1.1]
  \node[circle, draw, fill=black, inner sep=2pt, label=left:{$2n_1$}] (A) at (0,0) {};
  \node[circle, draw, fill=black, inner sep=2pt, label=right:{$2n_2$}] (B) at (1.0,0) {};
  \node[circle, draw, fill=black, inner sep=2pt, label={$2n_3$}] (C) at (0.5,0.866025) {};

  \draw (A) -- (B) node [midway, below] {$e_1$};
  \draw (B) -- (C) node [midway, right] {$e_2$};
  \draw (C) -- (A) node [midway, left] {$e_3$};
\end{tikzpicture} &
$S_3\cong D_3$\\

\hline

$4$ &
path &
\begin{tikzpicture}[scale=1.1]
  \node[circle, draw, fill=black, inner sep=2pt, label={$2n_1$}] (A) at (0,0) {};
  \node[circle, draw, fill=black, inner sep=2pt, label={$2n_2$}] (B) at (0.8,0) {};
  \node[circle, draw, fill=black, inner sep=2pt, label={$2n_3$}] (C) at (1.6,0) {};
  \node[circle, draw, fill=black, inner sep=2pt, label={$2n_4$}] (D) at (2.4,0) {};

  \draw (A) -- (B) node [midway, below] {$e_1$};
  \draw (B) -- (C) node [midway, below] {$e_2$};
  \draw (C) -- (D) node [midway, below] {$e_3$};
\end{tikzpicture} &
$\mathbb Z_2$\\

$4$ &
star &
\begin{tikzpicture}[scale=1.1]
  \node[circle, draw, fill=black, inner sep=2pt, label=left:{$2n_1$}] (A) at (0,0) {};
  \node[circle, draw, fill=black, inner sep=2pt, label=below:{$2n_4$}] (B) at (0.8,0) {};
  \node[circle, draw, fill=black, inner sep=2pt, label=right:{$2n_2$}] (C) at (1.6,0) {};
  \node[circle, draw, fill=black, inner sep=2pt, label=above:{$2n_3$}] (D) at (0.8,0.8) {};

  \draw (A) -- (B) node [midway, above] {$e_1$};
  \draw (B) -- (C) node [midway, below] {$e_2$};
  \draw (B) -- (D) node [midway, right] {$e_3$};
\end{tikzpicture} &
$S_3\cong D_3$\\

$4$ &
cycle &
\begin{tikzpicture}[scale=1.1]
  \node[circle, draw, fill=black, inner sep=2pt, label=left:{$2n_1$}] (A) at (0,0) {};
  \node[circle, draw, fill=black, inner sep=2pt, label=right:{$2n_2$}] (B) at (0.8,0) {};
  \node[circle, draw, fill=black, inner sep=2pt, label=right:{$2n_3$}] (C) at (0.8,0.8) {};
  \node[circle, draw, fill=black, inner sep=2pt, label=left:{$2n_4$}] (D) at (0,0.8) {};

  \draw (A) -- (B) node [midway, below] {$e_1$};
  \draw (B) -- (C) node [midway, right] {$e_2$};
  \draw (C) -- (D) node [midway, above] {$e_3$};
  \draw (D) -- (A) node [midway, left] {$e_4$};
\end{tikzpicture} &
$D_4$\\

$4$ &
path with pendant &
\begin{tikzpicture}[scale=1.1]
  \node[circle, draw, fill=black, inner sep=2pt, label={$2n_1$}] (A) at (0,0) {};
  \node[circle, draw, fill=black, inner sep=2pt, label={$2n_2$}] (B) at (-0.866025,0.5) {};
  \node[circle, draw, fill=black, inner sep=2pt, label=below:{$2n_3$}] (C) at (-0.866025, 0-0.5) {};
  \node[circle, draw, fill=black, inner sep=2pt, label={$2n_4$}] (D) at (1,0) {};

  \draw (A) -- (B) node [midway, below] {$e_1$};
  \draw (B) -- (C) node [midway, left] {$e_2$};
  \draw (C) -- (A) node [midway, below] {$e_3$};
  \draw (A) -- (D) node [midway, below] {$e_4$};
\end{tikzpicture}&
$\mathbb Z_2$\\

$4$ &
square with diagonal &
\begin{tikzpicture}[scale=1.1]
  \node[circle, draw, fill=black, inner sep=2pt, label=left:{$2n_1$}] (A) at (0,0) {};
  \node[circle, draw, fill=black, inner sep=2pt, label=right:{$2n_2$}] (B) at (0.8,0) {};
  \node[circle, draw, fill=black, inner sep=2pt, label=right:{$2n_3$}] (C) at (0.8,0.8) {};
  \node[circle, draw, fill=black, inner sep=2pt, label=left:{$2n_4$}] (D) at (0,0.8) {};

  \draw (A) -- (B) node [midway, below] {$e_1$};
  \draw (B) -- (C) node [midway, right] {$e_2$};
  \draw (C) -- (D) node [midway, above] {$e_3$};
  \draw (D) -- (A) node [midway, left] {$e_4$};
  \draw (A) -- (C) node [midway, below] {$e_5$};
\end{tikzpicture} &
$\mathbb Z_2 \times \mathbb Z_2$ \\

$4$ &
complete graph &
\begin{tikzpicture}[scale=1.1]
  \node[circle, draw, fill=black, inner sep=2pt, label=left:{$2n_1$}] (A) at (0,0) {};
  \node[circle, draw, fill=black, inner sep=2pt, label=right:{$2n_2$}] (B) at (0.8,0) {};
  \node[circle, draw, fill=black, inner sep=2pt, label=right:{$2n_3$}] (C) at (0.8,0.8) {};
  \node[circle, draw, fill=black, inner sep=2pt, label=left:{$2n_4$}] (D) at (0,0.8) {};

  \draw (A) -- (B) node [midway, below] {$e_1$};
  \draw (B) -- (C) node [midway, right] {$e_2$};
  \draw (C) -- (D) node [midway, above] {$e_3$};
  \draw (D) -- (A) node [midway, left] {$e_4$};
  \draw (A) -- (C) node [midway, below] {$e_5$};
  \draw (B) -- (D) node [midway, above] {$e_6$};

\end{tikzpicture}&
$S_4$ \\

\hline
\end{tabular}
\caption{
Classes of connected weighted graphs with corresponding symmetry groups. The symmetry groups listed are the automorphism groups of the underlying unweighted graphs; the symmetry group of a weighted graph is the corresponding stabilizer subgroup.
}
\label{tab:weighter_graphs}
\end{table}
\newpage
\begin{table}[H]
    \centering
    \begin{tabular}{|>{\centering\arraybackslash}m{25mm}|>{\centering\arraybackslash}m{50mm}|>{\centering\arraybackslash}m{35mm}|}
        \hline
        $\vec{\beta}$ &  Graphs & $S_{\vec{\beta}}$ \\
        \hline
        $[0,0,0,0,0,1]$ & \begin{tikzpicture}[scale=0.4, 
                    baseline=(current bounding box.center),
                    execute at end picture={
                        \path[use as bounding box]
                        ([yshift=-10pt]current bounding box.south west)
                        rectangle
                        ([yshift=10pt]current bounding box.north east);
                    }]

\begin{scope}[xshift=0mm]
  \node[circle, draw, fill=black, inner sep=2pt] (A) at (0, 0) {};

  \draw (A) to[out=120,in=180,loop] ();
  \draw (A) to[out=240,in=300,loop] ();
  \draw (A) to[out=0,in=60,loop] ();
\end{scope}

\end{tikzpicture} & $S_6$ \\
        \hline
        $[0,0,2]$       & \begin{tikzpicture}[scale=0.4, 
                    baseline=(current bounding box.center),
                    execute at end picture={
                        \path[use as bounding box]
                        ([yshift=-10pt]current bounding box.south west)
                        rectangle
                        ([yshift=10pt]current bounding box.north east);
                    }]

\begin{scope}[xshift=0mm]
  \node[circle, draw, fill=black, inner sep=2pt] (A) at (0, 0) {};
  \node[circle, draw, fill=black, inner sep=2pt] (B) at (1, 0) {};

  \draw (A) to (B);
  \draw (A) to[out=60,in=120,loop] ();
  \draw (B) to[out=60,in=120,loop] ();
\end{scope}

\begin{scope}[xshift=35mm]
  \node[circle, draw, fill=black, inner sep=2pt] (A) at (0, 0) {};
  \node[circle, draw, fill=black, inner sep=2pt] (B) at (1.5, 0) {};

  \draw (A) to[bend left=18] (B);
  \draw (A) to[bend right=18] (B);
  \draw (A) to (B);
\end{scope}

\end{tikzpicture} & $(S_3\times S_3)\rtimes S_2$\\
        \hline
        $[0,1,0,1]$     & \begin{tikzpicture}[scale=0.4, 
                    baseline=(current bounding box.center),
                    execute at end picture={
                        \path[use as bounding box]
                        ([yshift=-10pt]current bounding box.south west)
                        rectangle
                        ([yshift=10pt]current bounding box.north east);
                    }]

\begin{scope}[xshift=0mm]
  \node[circle, draw, fill=black, inner sep=2pt] (A) at (0, 0) {};
  \node[circle, draw, fill=black, inner sep=2pt] (B) at (1, 0) {};

  \draw (A) to[out=60,in=120,loop] ();
  \draw (B) to[out=60,in=120,loop] ();
  \draw (B) to[out=240,in=300,loop] ();
\end{scope}

\begin{scope}[xshift=35mm]
  \node[circle, draw, fill=black, inner sep=2pt] (A) at (0, 0) {};
  \node[circle, draw, fill=black, inner sep=2pt] (B) at (1, 0) {};

  \draw (A) to[bend left=15] (B);
  \draw (A) to[bend right=15] (B);
  \draw (B) to[out=60,in=120,loop] ();
\end{scope}

\end{tikzpicture} & $S_2\times S_4$\\
        \hline
        $[1,0,0,0,1]$   & \begin{tikzpicture}[scale=0.4, 
                    baseline=(current bounding box.center),
                    execute at end picture={
                        \path[use as bounding box]
                        ([yshift=-10pt]current bounding box.south west)
                        rectangle
                        ([yshift=10pt]current bounding box.north east);
                    }]

\begin{scope}[xshift=0mm]
  \node[circle, draw, fill=black, inner sep=2pt] (A) at (0, 0) {};
  \node[circle, draw, fill=black, inner sep=2pt] (B) at (1, 0) {};

  \draw (A) to (B);
  \draw (B) to[out=30,in=90,loop] ();
  \draw (B) to[out=270,in=330,loop] ();
\end{scope}

\end{tikzpicture} & $S_5$\\
        \hline
        $[0,3]$         & \begin{tikzpicture}[scale=0.4, 
                    baseline=(current bounding box.center),
                    execute at end picture={
                        \path[use as bounding box]
                        ([yshift=-10pt]current bounding box.south west)
                        rectangle
                        ([yshift=10pt]current bounding box.north east);
                    }]

\begin{scope}[xshift=0cm]
  \node[circle, draw, fill=black, inner sep=2pt] (A) at (0, 1) {};
  \node[circle, draw, fill=black, inner sep=2pt] (B) at (1.73, 1) {};
  \node[circle, draw, fill=black, inner sep=2pt] (C) at (0.86, 0) {};

  \draw (A) to[bend left=12] (B);
  \draw (A) to[bend right=12] (B);
  \draw (C) to[out=60,in=120,loop] ();
\end{scope}

\begin{scope}[xshift=35mm]
  \node[circle, draw, fill=black, inner sep=2pt] (A) at (0, 1) {};
  \node[circle, draw, fill=black, inner sep=2pt] (B) at (1.73, 1) {};
  \node[circle, draw, fill=black, inner sep=2pt] (C) at (0.86, 0) {};

  \draw (A) to (B);
  \draw (B) to (C);
  \draw (C) to (A);
\end{scope}

\begin{scope}[xshift=70mm]
  \node[circle, draw, fill=black, inner sep=2pt] (A) at (0, 1) {};
  \node[circle, draw, fill=black, inner sep=2pt] (B) at (1.73, 1) {};
  \node[circle, draw, fill=black, inner sep=2pt] (C) at (0.86, 0) {};

  \draw (A) to[out=60,in=120,loop] ();
  \draw (B) to[out=60,in=120,loop] ();
  \draw (C) to[out=60,in=120,loop] ();
\end{scope}

\end{tikzpicture} & $(S_2\times S_2\times S_2)\rtimes S_3$\\
        \hline
        $[1,1,1]$       & \begin{tikzpicture}[scale=0.4, 
                    baseline=(current bounding box.center),
                    execute at end picture={
                        \path[use as bounding box]
                        ([yshift=-10pt]current bounding box.south west)
                        rectangle
                        ([yshift=10pt]current bounding box.north east);
                    }]

\begin{scope}[xshift=0mm]
  \node[circle, draw, fill=black, inner sep=2pt] (A) at (0, 0) {};
  \node[circle, draw, fill=black, inner sep=2pt] (B) at (1, 0) {};
  \node[circle, draw, fill=black, inner sep=2pt] (C) at (2, 0) {};

  \draw (A) to (B);
  \draw (B) to (C);
  \draw (C) to[out=60,in=120,loop] ();
\end{scope}

\begin{scope}[xshift=35mm]
  \node[circle, draw, fill=black, inner sep=2pt] (A) at (0, 0) {};
  \node[circle, draw, fill=black, inner sep=2pt] (B) at (1, 0) {};
  \node[circle, draw, fill=black, inner sep=2pt] (C) at (2, 0) {};

  \draw (A) to (B);
  \draw (B) to[bend left=15] (C);
  \draw (B) to[bend right=15] (C);
\end{scope}

\begin{scope}[xshift=70mm]
  \node[circle, draw, fill=black, inner sep=2pt] (A) at (0, 0) {};
  \node[circle, draw, fill=black, inner sep=2pt] (B) at (1, 0) {};
  \node[circle, draw, fill=black, inner sep=2pt] (C) at (2, 0) {};

  \draw (A) to (B);
  \draw (B) to[out=60,in=120,loop] ();
  \draw (C) to[out=60,in=120,loop] ();
\end{scope}

\end{tikzpicture} & $S_2\times S_3$\\
        \hline
        $[2,0,0,1]$     & \begin{tikzpicture}[scale=0.4, 
                    baseline=(current bounding box.center),
                    execute at end picture={
                        \path[use as bounding box]
                        ([yshift=-10pt]current bounding box.south west)
                        rectangle
                        ([yshift=10pt]current bounding box.north east);
                    }]

\begin{scope}[xshift=0mm]
  \node[circle, draw, fill=black, inner sep=2pt] (A) at (0, 1) {};
  \node[circle, draw, fill=black, inner sep=2pt] (B) at (1.4, 1) {};
  \node[circle, draw, fill=black, inner sep=2pt] (C) at (0.7, 0) {};

  \draw (A) to (B);
  \draw (C) to[out=-30,in=30,loop] ();
  \draw (C) to[out=150,in=210,loop] ();
\end{scope}

\begin{scope}[xshift=35mm]
  \node[circle, draw, fill=black, inner sep=2pt] (A) at (0, 1) {};
  \node[circle, draw, fill=black, inner sep=2pt] (B) at (0, 0) {};
  \node[circle, draw, fill=black, inner sep=2pt] (C) at (1.2, 0.5) {};

  \draw (A) to (C);
  \draw (B) to (C);
  \draw (C) to[out=60,in=120,loop] ();
\end{scope}

\end{tikzpicture} & $S_2\times S_4$\\
        \hline
        $[2,2]$         & \begin{tikzpicture}[scale=0.4, 
                    baseline=(current bounding box.center),
                    execute at end picture={
                        \path[use as bounding box]
                        ([yshift=-10pt]current bounding box.south west)
                        rectangle
                        ([yshift=10pt]current bounding box.north east);
                    }]

\begin{scope}[xshift=0mm]
  \node[circle, draw, fill=black, inner sep=2pt] (A) at (0,0) {};
  \node[circle, draw, fill=black, inner sep=2pt] (B) at (1,0) {};
  \node[circle, draw, fill=black, inner sep=2pt] (C) at (0,1) {};
  \node[circle, draw, fill=black, inner sep=2pt] (D) at (1,1) {};

  \draw (A) to (B);
  \draw (C) to[out=60,in=120,loop] ();
  \draw (D) to[out=60,in=120,loop] ();
\end{scope}

\begin{scope}[xshift=35mm]
  \node[circle, draw, fill=black, inner sep=2pt] (A) at (0,0) {};
  \node[circle, draw, fill=black, inner sep=2pt] (B) at (1,0) {};
  \node[circle, draw, fill=black, inner sep=2pt] (C) at (0,1) {};
  \node[circle, draw, fill=black, inner sep=2pt] (D) at (1,1) {};

  \draw (C) to[bend left=15] (D);
  \draw (C) to[bend right=15] (D);
  \draw (A) to (B);
\end{scope}

\begin{scope}[xshift=70mm]
  \node[circle, draw, fill=black, inner sep=2pt] (A) at (0,0) {};
  \node[circle, draw, fill=black, inner sep=2pt] (B) at (1,0) {};
  \node[circle, draw, fill=black, inner sep=2pt] (C) at (0,1) {};
  \node[circle, draw, fill=black, inner sep=2pt] (D) at (1,1) {};

  \draw (B) to (D);
  \draw (A) to (C);
  \draw (C) to (D);
\end{scope}

\begin{scope}[xshift=105mm]
  \node[circle, draw, fill=black, inner sep=2pt] (A) at (0,0) {};
  \node[circle, draw, fill=black, inner sep=2pt] (B) at (1,0) {};
  \node[circle, draw, fill=black, inner sep=2pt] (C) at (0,1) {};
  \node[circle, draw, fill=black, inner sep=2pt] (D) at (1,1) {};

  \draw (A) to (C);
  \draw (B) to (C);
  \draw (D) to[out=60,in=120,loop] ();
\end{scope}

\end{tikzpicture} & $S_2\times((S_2 \times S_2)\rtimes S_2)$\\
        \hline
        $[3,0,1]$       & \begin{tikzpicture}[scale=0.4, 
                    baseline=(current bounding box.center),
                    execute at end picture={
                        \path[use as bounding box]
                        ([yshift=-10pt]current bounding box.south west)
                        rectangle
                        ([yshift=10pt]current bounding box.north east);
                    }]

\begin{scope}[xshift=0mm]
  \node[circle, draw, fill=black, inner sep=2pt] (A) at (0, 1.2) {};
  \node[circle, draw, fill=black, inner sep=2pt] (B) at (1.4, 1.2) {};
  \node[circle, draw, fill=black, inner sep=2pt] (C) at (0, 0) {};
  \node[circle, draw, fill=black, inner sep=2pt] (D) at (1.4, 0) {};

  \draw (A) to (B);
  \draw (C) to (D);
  \draw (D) to[out=60,in=120,loop] ();
\end{scope}

\begin{scope}[xshift=35mm]
  \node[circle, draw, fill=black, inner sep=2pt] (A) at (0, 1) {};
  \node[circle, draw, fill=black, inner sep=2pt] (B) at (1.2, 0.5) {};
  \node[circle, draw, fill=black, inner sep=2pt] (C) at (0, 0) {};
  \node[circle, draw, fill=black, inner sep=2pt] (D) at (2.4, 0.5) {};

  \draw (A) to (B);
  \draw (C) to (B);
  \draw (B) to (D);
\end{scope}

\end{tikzpicture} & $S_3 \times S_3$\\
        \hline
        $[4,1]$         & \begin{tikzpicture}[scale=0.4, 
                    baseline=(current bounding box.center),
                    execute at end picture={
                        \path[use as bounding box]
                        ([yshift=-10pt]current bounding box.south west)
                        rectangle
                        ([yshift=10pt]current bounding box.north east);
                    }]

\begin{scope}[xshift=0mm]
  \node[circle, draw, fill=black, inner sep=2pt] (A) at (0, 1) {};
  \node[circle, draw, fill=black, inner sep=2pt] (B) at (1.6, 1) {};
  \node[circle, draw, fill=black, inner sep=2pt] (C) at (-0.2, 0) {};
  \node[circle, draw, fill=black, inner sep=2pt] (D) at (0.8, 0) {};
  \node[circle, draw, fill=black, inner sep=2pt] (E) at (1.8, 0) {};

  \draw (A) to (B);
  \draw (C) to (D);
  \draw (D) to (E);
\end{scope}

\begin{scope}[xshift=35mm]
  \node[circle, draw, fill=black, inner sep=2pt] (A) at (0, 1) {};
  \node[circle, draw, fill=black, inner sep=2pt] (B) at (1.2, 1) {};
  \node[circle, draw, fill=black, inner sep=2pt] (C) at (0, 0) {};
  \node[circle, draw, fill=black, inner sep=2pt] (D) at (1.2, 0) {};
  \node[circle, draw, fill=black, inner sep=2pt] (E) at (2, 0.5) {};

  \draw (A) to (B);
  \draw (C) to (D);
  \draw (E) to[out=60,in=120,loop] ();
\end{scope}

\end{tikzpicture} & $S_4 \times S_2$\\
        \hline
        $[6]$           & \begin{tikzpicture}[scale=0.4, 
                    baseline=(current bounding box.center),
                    execute at end picture={
                        \path[use as bounding box]
                        ([yshift=-10pt]current bounding box.south west)
                        rectangle
                        ([yshift=10pt]current bounding box.north east);
                    }]

\begin{scope}[xshift=0mm]
  \node[circle, draw, fill=black, inner sep=2pt] (A) at (0, 1) {};
  \node[circle, draw, fill=black, inner sep=2pt] (B) at (0.8, 1) {};
  \node[circle, draw, fill=black, inner sep=2pt] (C) at (1.6, 1) {};
  \node[circle, draw, fill=black, inner sep=2pt] (D) at (0, 0) {};
  \node[circle, draw, fill=black, inner sep=2pt] (E) at (0.8, 0) {};
  \node[circle, draw, fill=black, inner sep=2pt] (F) at (1.6, 0) {};

  \draw (A) to (D);
  \draw (B) to (E);
  \draw (C) to (F);
\end{scope}

\end{tikzpicture} & $S_6$\\
        \hline
    \end{tabular}
    \caption{All degree vectors $\vec\beta$ with $\|\vec\beta\|_E=3$, the corresponding graph representatives, and the symmetry group $S_{\vec\beta}\subset S_6$. For $\vec\beta=(\beta_1,\beta_2,\ldots)$, the group is $ _{\vec\beta}=\prod_{j\ge 1}\left(S_j^{\beta_j}\rtimes S_{\beta_j}\right)$, where $S_j^{\beta_j}$ permutes the slots within the $\beta_j$ vertices of degree $j$, and $S_{\beta_j}$ permutes those vertices among themselves.}
    \label{tab:degreeVectors}
\end{table}

\end{document}